\documentclass[11pt]{article}
\usepackage{amsmath}
\usepackage{amsfonts}
\usepackage{amsthm}
\usepackage{graphics}
\usepackage{epsfig} 
\usepackage{amssymb}
\usepackage{amscd}
\usepackage[all]{xy}
\usepackage{latexsym}
\usepackage{graphicx}
\usepackage{multirow}
\usepackage{geometry}
\newtheorem{thm}{Theorem}[section]

\newtheorem{prop}[thm]{Proposition}

\newtheorem{Def}[thm]{Definition}
\newtheorem{rem}[thm]{Remark}

\newtheorem{case}{Case}

\numberwithin{equation}{section}
\numberwithin{figure}{section}

\def\rchi{{\hbox{\raise1.5pt\hbox{$\chi$}}}}
\def\Aut{{\text{\rm{Aut}}}}

\def\isom{\cong}
\def\tensor{\otimes}

\def\Ker{{\text{\rm{Ker}}}}

\def\a{\alpha}
\def\b{\beta}
\def\lam{\lambda}

\def\vol{{\text{\rm{vol}}}}

\def\Sym{{\text{\rm{Sym}}}}

\newcommand{\Mbar}{{\overline{\mathcal{M}}}}

\newcommand{\bP}{{\mathbb{P}}}
\newcommand{\bC}{{\mathbb{C}}}

\newcommand{\bL}{{\mathbb{L}}}

\newcommand{\bR}{{\mathbb{R}}}
\newcommand{\bT}{{\mathbb{T}}}
\newcommand{\bZ}{{\mathbb{Z}}}
\newcommand{\cA}{{\mathcal{A}}}

\newcommand{\cM}{{\mathcal{M}}}

\newcommand{\cL}{{\mathcal{L}}}
\newcommand{\cO}{{\mathcal{O}}}

\newcommand{\la}{{\langle}}
\newcommand{\ra}{{\rangle}}
\newcommand{\half}{{\frac{1}{2}}}
\newcommand{\bp}{{\mathbf{p}}}
\newcommand{\bx}{{\mathbf{x}}}
\newcommand{\wL}{{\widehat{L}}}
\newcommand{\wV}{{\widehat{V}}}

\title{The Kontsevich constants for the volume of the moduli of curves
  and topological recursion}

\author{ Kevin M. Chapman \\ 
  Department of Mathematics\\
  University of California\\
  Davis, CA 95616--8633, U.S.A. \\
  \texttt{kmchapman@math.ucdavis.edu}
  \and
  Motohico Mulase \\
  Department of Mathematics\\
  University of California\\
  Davis, CA 95616--8633, U.S.A. \\
  \texttt{mulase@math.ucdavis.edu}
  \and
  Brad Safnuk
Department of Mathematics\\
Central Michigan University\\
Mount Pleasant, MI 48859, U.S.A. \\
\texttt{brad.safnuk@cmich.edu}
}

\begin{document}
  \maketitle

\begin{abstract}
We give an Eynard-Orantin type
\emph{topological recursion formula}
for the canonical Euclidean volume 
of the combinatorial moduli space of pointed 
smooth algebraic curves.
The recursion comes from
the edge removal
operation on the space of ribbon graphs. As an application 
we obtain a new proof of the Kontsevich constants for the
ratio of the
Euclidean and the symplectic volumes of the moduli space
of curves.

MSC Primary: 14N35, 05C30, 53D30, 11P21; Secondary: 81T30
\end{abstract}

\allowdisplaybreaks

\tableofcontents

\section{Introduction}
\label{sect:intro}

The purpose of this paper is to identify a 
combinatorial origin of the
\emph{topological recursion formula} of Eynard and Orantin
\cite{EO1} as the operation of edge removal from a ribbon 
graph. As an application  of our formalism, we establish a
new proof of the
formula for the Kontsevich constants 
$\rho = 2^{5g-5+2n}$ of
\cite[Appendix~C]{K1992}.

In moduli theory it often happens that we have two 
different notions of the \emph{volume} of the moduli
space.
The volume may be defined by
 the push-forward measure of the 
canonical construction of the moduli space. 
Or it may be defined as the symplectic
volume with respect to  
 the intrinsic symplectic structure of the moduli space.
 An example of such situations is the moduli space of flat 
$G$-bundles on a fixed Riemann surface
 for a compact Lie group $G$  \cite{JK, JW, L, W1991b}.
 In this case, the two definitions of the volume 
 agree.

The space we study in this paper is the 
combinatorial model of moduli space 
$\cM_{g,n}$ of smooth algebraic curves of genus $g$ with
$n$ distinct marked points. It also has two 
different families of volumes parametrized by
$n$ positive real parameters.
One comes from the push-forward measure,
and the other comes from the intrinsic symplectic structure
depending on these parameters. And again these two 
notions of volume agree.

The moduli space $\cM_{g,n}$
 admits orbifold cell-decompositions
parametrized by the collection of
positive real numbers assigned to the marked points.
This orbifold is identified as the space of ribbon graphs of
a prescribed perimeter length, using the theory of 
Strebel differentials.
In his seminal paper of 1992, Kontsevich \cite{K1992}
calculated the 
symplectic volume of orbi-cells, and compared it with the
standard Euclidean volume. He found that the ratio was a constant
depending only on the genus of the curve and the number of 
marked points. This constant plays a crucial role in his
 \emph{main identity}, and hence in his
proof of the Witten conjecture.
 He wrote in Appendix~C of \cite{K1992}  that his
 proof of the evaluation of this constant
 ``presented here is \emph{not nice},
but we don't know any other proof.'' In this article we give
another proof of the formula for the Kontsevich constant,
based on 
the topological recursion for ribbon graphs.

The idea of topological recursion has been used as
an effective tool for calculating many quantities
related to the moduli space $\cM_{g,n}$ and its 
Deligne-Mumford compactification
$\Mbar_{g,n}$. The quantities
we can deal with include tautological intersection numbers
and certain Gromov-Witten invariants. 
Suppose we have a collection of quantities
 $v_{g,n}$ for $g\ge 0$ and 
$n>0$ subject to the \emph{stability} condition $2g-2+n>0$,
which guarantees the finiteness of the automorphism group
of an element of $\cM_{g,n}$.
By an Eynard-Orantin type topological recursion formula
\cite{EO1} we mean a particular
inductive formula for $v_{g,n}$ with respect to the 
\emph{complexity} $2g-2+n$ of the form
\begin{equation}
\label{eq:TR}
v_{g,n} = f_1(v_{g,n-1}) + f_2(v_{g-1, n+1})
+ \sum_{\substack{g_1+g_2=g\\n_1+n_2=n-1}} ^{\text{stable}}
f_3(v_{g_1,n_1+1},v_{g_1,n_2+1})
\end{equation}
with linear operators $f_1, f_2$ and  a bilinear operator $f_3$,
where the sum is taken for all possible partitions of $g$ and $n-1$ 
subject to the stability conditions $2g_1 -1 +n_1>0$ and
$2g_2-1+n_2>0$. We refer to Section~\ref{sect:EO} for 
more detail.

\begin{figure}[htb]
\centerline{\epsfig{file=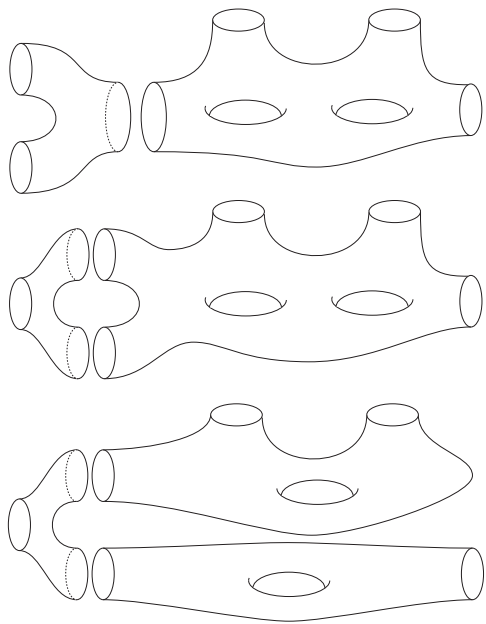, width=2.5in}}
\caption{The topological recursion. The reduction of
$2g-2+n$ by $1$ corresponds 
to cutting off of a pair of pants from 
an $n$-punctured surface.}
\label{fig:recursion}
\end{figure}

 There are many examples of
such formulas. 
\begin{enumerate}
\item The Witten-Kontsevich theory
for
the tautological cotangent class (i.e.\ the $\psi$-class) intersection numbers
\begin{equation}
\label{eq:tau}
\la \tau_{d_1}\cdots\tau_{d_n}\ra_{g,n}
=
\int_{\Mbar_{g,n}} c_1(\bL_1) ^{d_1}\cdots c_1(\bL_n) ^{d_n}
\end{equation}
on the moduli stack $\Mbar_{g,n}$
of stable algebraic curves of genus
$g$ with $n$ distinct smooth marked points. 
The Dijkgraaf-Verlinde-Verlinde formula \cite{DVV},
 which is equivalent to the
Virasoro constraint condition, is a topological recursion
of the form (\ref{eq:TR}).

\item The Mirzakhani recursion formula for the Weil-Petersson 
volume of the moduli space of bordered hyperbolic surfaces
with prescribed geodesic boundaries \cite{Mir1, Mir2}
is a topological recursion.

\item Mixed intersection numbers
$$
 \la \tau_{d_1}\cdots\tau_{d_n} \kappa_1 ^{m_1}
\kappa_2 ^{m_2} \kappa_3 ^{m_3}\cdots\ra_{g,n}
$$
of $\psi$-classes and the Mumford-Morita-Miller $\kappa$-classes
satisfy a topological recursion, first found in \cite{MS} for
the case with $\kappa_1$ and later generalized in \cite{LX}.

\item The expectation values of the product of resolvents of 
various matrix models satisfy a topological recursion 
(see for example, \cite{E2004}).
This is the origin of the work \cite{EO1}.

\item
Indeed, the first three geometric  theories 
turned out to be examples of the
general theory \cite{EO1} of
topological recursions \cite{E2007, EO2}, though 
geometric theories had been discovered
earlier than the publication of \cite{EO1}.

\item Both open and closed Gromov-Witten invariants of 
an arbitrary toric Calabi-Yau threefold are expected to satisfy a
topological recursion. This is the \emph{remodeling}
conjecture of \cite{M, BKMP}. 

\item Simple Hurwitz numbers satisfy a topological 
recursion. It was first conjectured in \cite{BM} based on 
a limit case of the remodeling conjecture, 
and was recently proved in \cite{BEMS,EMS, MZ}.

\item
The simplest case of the remodeling conjecture for
 $\bC^3$ was proved in \cite{Chen, Z1,Z2} based on the
 Laplace transform technique of \cite{EMS}.
 
 \item As shown below,
 the number of metric ribbon graphs with integer edge
 lengths for a prescribed boundary condition satisfies 
 a topological recursion.

\end{enumerate}
Our current paper provides an elementary approach to the idea
of topological recursion that uniformly 
explains the combinatorial nature
of the geometric examples (1), (2), (3), (7), (8) and (9).

The work of Harer \cite{Harer}, 
Mumford \cite{Mumford}, Strebel \cite{Strebel}, Thurston 
and others  \cite{STT}
show that there is a topological orbifold isomorphism 
\begin{equation*}
\cM_{g,n}\times \bR_+ ^n \isom RG_{g,n},
\end{equation*}
where
\begin{equation*}
 RG_{g,n} = \coprod_{\substack{\Gamma {\text{ ribbon graph}}\\
{\text{of type }} (g,n)}}\frac{\bR_+ ^{e(\Gamma)}}{\Aut(\Gamma)}
 \end{equation*}
is the space of metric ribbon graphs
of genus $g$ and $n$ boundary components, and
$e(\Gamma)$ is the number of edges of a ribbon graph $\Gamma$.
We denote by $\pi:RG_{g,n}\longrightarrow \bR_+ ^n$
the natural projection, and its fiber at $\bp\in\bR_+ ^n$
by $RG_{g,n}(\bp) = \pi^{-1}(\bp)$.
To give a combinatorial description of
 tautological intersection numbers 
 (\ref{eq:tau}) on $\Mbar_{g,n}$,
 Kontsevich \cite[Page~8]{K1992}
introduced a combinatorial symplectic form
$\omega_K(\bp)$ on $RG_{g,n}(\bp) \isom \cM_{g,n}$
and its \emph{symplectic} volume 
\begin{equation}
\label{eq:vS}
v_{g,n} ^S (\bp) = \int_{RG_{g,n}(\bp)}\exp\big( \omega_K(\bp)
\big).
\end{equation}
The definition of this symplectic form is given in
Section~\ref{sect:symp}.
At each orbi-cell level, the derivative $d\pi$ of 
the projection map $\pi$ is determined
by the edge-face incidence matrix
$$
A_\Gamma:\bR_+ ^{e(\Gamma)}\longrightarrow
\bR_+ ^n
$$
of a ribbon graph $\Gamma$. Note that we have
the standard volume forms $d\ell_1\wedge \cdots\wedge 
d\ell_{e(\Gamma)}$ on
$\bR_+ ^{e(\Gamma)}$ 
and $dp_1\wedge\cdots\wedge dp_n$ on $\bR_+ ^n$.
We can define the Euclidean volume
of the inverse image $P_\Gamma(\bp)=
A_{\Gamma}^{-1}(\bp)$ of $\bp\in\bR_+^n$ 
using the push-forward measure
by
$$
\vol(P_\Gamma (\bp))
=\left. \frac{(A_\Gamma)_*
(d\ell_1\wedge \cdots\wedge d\ell_{e(\Gamma)})}
{dp_1\wedge\cdots\wedge dp_n}\right|_\bp,
$$
where $(A_\Gamma)_*(d\ell_1\wedge \cdots\wedge 
d\ell_{e(\Gamma)})$ is
the $n$-form on $\bR_+ ^n$ obtained by integrating 
the volume form on $\bR_+ ^{e(\Gamma)}$ 
along the fiber $A_\Gamma ^{-1}(\bp)$. 
The \emph{Euclidean} volume of the moduli space is defined by
\begin{equation*}
v_{g,n}^E(\bp)=
\sum_{\substack{\Gamma {\text{ ribbon graph}}\\
{\text{of type }} (g,n)}} \frac{\vol(P_\Gamma(\mathbf{p}))}
{|\Aut(\Gamma)|}.
\end{equation*}
In Appendix~C of \cite{K1992}, Kontsevich proved the following.

\begin{thm}[\cite{K1992}]
The ratio of the symplectic volume and the Euclidean volume
of $RG_{g,n}(\bp)$ 
is a constant depending only on $g$ and $n$, and its value is
\begin{equation}
\label{eq:intro-rho}
\rho = \frac{v_{g,n} ^S (\bp)}{v_{g,n} ^E (\bp)} = 
2^{5g-5+2n}.
\end{equation}
\end{thm}

\begin{rem}
The Euclidean volume of the polytope 
$$
P_\Gamma (\bp)= \{\bx\in \bR_+ ^{e(\Gamma)}
\;|\; A_\Gamma \bx = \bp\}
$$
is a \emph{quasi-polynomial}  and is difficult to calculate
in general. It is 
quite surprising that the ratio $\rho$ of the two functions is indeed
a constant. Although he says ``not nice,''  Kontsevich's 
original proof is
a beautiful application of homological algebra to the
complexes defined by the incidence matrix $A_\Gamma$.
\end{rem}

The new proof we present here uses an elementary
argument on the topological recursion of ribbon graphs
corresponding to the edge removal operation.
We show that both $v_{g,n} ^S(\bp)$ and 
$2^{5g-5+2n}\cdot v_{g,n} ^E(\bp)$ satisfy
exactly the same induction formula based on $2g-2+n$,
\emph{after taking the Laplace transform}. We then calculate 
the initial condition for the recursion formula, i.e., the cases
for $(g,n)=(0,3)$
and $(1,1)$, and see that the equality holds.
Since the topological recursion uniquely determines the
value for every $(g,n)$ subject to the stability condition
$2g-2+n>0$, we conclude that
$$
{v_{g,n} ^S (\bp)}=2^{5g-5+2n}\cdot {v_{g,n} ^E (\bp)}.
$$

Here the appearance of the Laplace transform is significant. 
The Laplace transform plays a mysterious as well as a crucial role in
each of the works \cite{E2007, EMS,EO2,K1992, MZ,OP1}.
In the light of the Eynard-Orantin recursion 
formalism \cite{EO1} and the remodeling conjecture
due to Mari\~no \cite{M} and
Bouchard-Klamm-Mari\~no-Pasquetti \cite{BKMP},
we find that the Laplace transform appearing in these 
contexts is the \emph{mirror map}.  
Usually mirror symmetry is considered as a duality, and hence
a family of Fourier-Mukai type transforms naturally appears
\cite{HT, SYZ}.
In our context, however, the nature of duality is not
apparent. 
On one side of 
the story (the A-model side) 
we have a combinatorial structure. The mirror 
symmetry transforms this combinatorial
structure into the world of
complex analysis (the B-model side). In the complex
analysis side we have such objects as
the residue calculus of \cite{EO1} and
integrable
nonlinear PDEs such as the KdV
equations \cite{K1992,LX,MS,W1991}, the KP hierarchy
\cite{Kazarian, KazarianLando, O2}, Frobenius manifold 
structures \cite{D, DZ}, the Ablowitz-Ladik hierarchy
\cite{Brini}, and more general integrable systems 
considered in \cite{FJMR, FJR1, FJR2}. The mathematical 
apparatus of the
mirror map hidden in these structures
is indeed the Laplace transform.

This paper is organized as follows.
In Section~\ref{sect:combinatorial}
we review ribbon graphs and combinatorial 
description  of the moduli space $\cM_{g,n}$
that are necessary for our investigation.
Although the definition of the Euclidean volume 
of $RG_{g,n}(\bp)$ is straightforward, it seems to be difficult to 
calculate it and there is no concrete  
formula. The approach we take in this paper is to 
appeal to the counting of \emph{lattice points}
of $RG_{g,n}(\bp)$. Thus Section~\ref{sect:lattice}
is devoted to proving an effective topological 
recursion formula for the number of lattice points
in the space of metric ribbon graphs with prescribed
perimeters. Our proof is based on counting ciliated
ribbon graphs.
Once we find the number of lattice points in 
$RG_{g,n}(\bp)$, we can obtain its volume
by taking the limit as the mesh of the lattice tends to $0$.
To compare the number of lattice points and the volume,
the simplest path is to take the Laplace transform.
Thus we are led to calculating the Laplace transform
of the topological recursion for the number of lattice points
in Section~\ref{sect:LTintegral}. 
After establishing the Laplace transform formula,
one can read off the information of the Euclidean volume 
of $RG_{g,n}(\bp)$ as the leading terms of the Laplace
transform,
by introducing the right coordinate system. 
This is carried out in Section~\ref{sect:Euc}.
The Kontsevich symplectic form is defined in 
Section~\ref{sect:symp}, and the topological 
recursion for the symplectic volume due to 
\cite{BCSW} is reviewed. With these preparations, we
give 
a new and simple proof of (\ref{eq:intro-rho}).
In Section~\ref{sect:EO} we explain the 
Eynard-Orantin formalism.  
This formalism is independent
on the context and provides the same formula. 
We then convert our recursion formulas into this
formalism, and observe how they all fit together in a 
single formula.
This is the
beauty and strength of the Eynard-Orantin formalism.

We present a full detail of the calculations of
the Laplace transform in this paper, hoping it may 
lead to a deeper understanding of the Eynard-Orantin 
theory and the mirror map. 
Appendix~\ref{app:LTProof}
is thus devoted to giving a proof of (\ref{eq:LTrecursion})
and (\ref{eq:LTSrecursion}).
These recursion formulas start with the initial values 
$(g,n)=(0,3)$ and $(g,n)=(1,1)$. The Eynard-Orantin theory
also uses the unstable case $(g,n)=(0,2)$. All these values 
are calculated in Appendix~\ref{app:examples}, together with
a few more examples.

\subsection*{Acknowledgement}
The authors thank the referee 
for important comments that improved 
the clarity of the paper. 
M.M.\ thanks Soheil Arabshahi, 
Zainal bin Abdul Aziz,
Minji Kim, and Jian Zhou for useful discussions.
He is also grateful to Michael Pankava and 
Andy Port for discussions on the
Laplace transform formulas.
During the preparation of this work,
the research of K.C.\  was supported by NSF grant DMS-0636297, 
M.M.\ received support from the American Institute 
of Mathematics, NSF, 
Universiti Teknologi Malaysia, and Tsinghua University in 
Beijing,
and the research of B.S.\ was supported 
by Central Michigan University.

\section{The combinatorial model of the moduli space}
\label{sect:combinatorial}

Let us begin with reviewing basic facts about ribbon graphs
and the combinatorial model of the moduli space 
$\cM_{g,n}$ due to Harer
\cite{Harer},
Mumford \cite{Mumford}, and Strebel 
\cite{Strebel}. We refer to \cite{MP1998} for precise 
definitions and more detailed exposition.

A \emph{ribbon graph} of topological type $(g,n)$
is the $1$-skeleton of a cell-decomposition of a closed
oriented topological surface $\Sigma$ of genus $g$
that decomposes the surface into a disjoint union of 
$v$ $0$-cells, $e$ $1$-cells, and $n$ $2$-cells. The Euler 
characteristic of the surface is given by $2-2g = v-e+n$. 
The $1$-skeleton of a cell-decomposition is a graph 
$\Gamma$ drawn on $\Sigma$, which consists
of $v$ vertices and $e$ edges. 
An edge can form a loop. We denote by
$\Sigma_\Gamma$ the cell-decomposed surface with $\Gamma$ 
its $1$-skeleton.
Alternatively, a ribbon graph can be
defined as a graph with a cyclic order
given to the incident half-edges at each vertex. By abuse of
terminology, we call the boundary of a $2$-cell of 
$\Sigma_\Gamma$ a
\emph{boundary}  of $\Gamma$, and the $2$-cell itself as
a \emph{face} of $\Gamma$.

A \emph{metric} ribbon graph is a ribbon graph with a
positive real number (the length) assigned to each edge. 
For a given ribbon graph $\Gamma$ with $e=e(\Gamma)$ 
edges, the space of metric
ribbon graphs is $\bR_+ ^{e(\Gamma)}/\Aut (\Gamma)$, 
where the
automorphism group acts through permutations of edges
(see \cite[Section~1]{MP1998}).
We restrict ourselves to the case that 
$\Aut (\Gamma)$ fixes each $2$-cell of the cell-decomposition.
If we also restrict that every vertex of a ribbon graph has degree
(i.e., valence) $3$ or more, then
using the canonical holomorphic coordinate system of 
a topological surface  \cite[Section~4]{MP1998}
 and the Strebel differentials \cite{Strebel}, we obtain
an isomorphism of topological orbifolds \cite{Harer,Mumford, STT}
\begin{equation}
\label{eq:M=RG}
{\cM}_{g,n}\times \bR_+ ^n \isom RG_{g,n}.
\end{equation}
Here
$$
RG_{g,n} = \coprod_{\substack{\Gamma {\text{ ribbon graph}}\\
{\text{of type }} (g,n)}} 
\frac{\bR_+ ^{e(\Gamma)}}{\Aut (\Gamma)}
$$
is the orbifold consisting of metric ribbon graphs of a given
topological type $(g,n)$ with degree $3$ or more. 
The degree condition is necessary to bound the number
of edges $e(\Gamma)$ for a given topological type $(g,n)$. 
If we allow degree $2$ vertices, then there are infinitely many
different ribbon graphs for every $(g,n)$. By restricting to 
 ribbon graphs 
of degree $3$ or more, we have the bound 
$e(\Gamma)\le 3(2g-2+n)$, which gives the dimension 
of each orbi-cell $\bR_+ ^{e(\Gamma)}/\Aut (\Gamma)$.
The gluing of orbi-cells
 is done by making 
the length of a non-loop edge tend to $0$. The space
 $RG_{g,n}$ is a smooth orbifold 
(see \cite[Section~3]{MP1998}, \cite{STT}). We denote by $\pi:RG_{g,n}\longrightarrow
\bR_+ ^n$ the natural projection via (\ref{eq:M=RG}), which
is the assignment of the collection of perimeter length
of each boundary to a given metric ribbon graph.

Take a ribbon graph $\Gamma$. Since $\Aut(\Gamma)$
fixes every boundary component of $\Gamma$, they can be labeled
by $N=\{1,2\dots,n\}$. For a moment let us give a label to 
each edge of $\Gamma$ from an index set $E = \{1,2,\dots,e\}$. 
The edge-face incidence matrix  is defined by
\begin{equation}
\label{eq:incidence}
\begin{aligned}
A_\Gamma &= \big[
a_{i\eta}\big]_{i\in N,\;\eta\in E};\\
a_{i\eta} &= \text{ the number of times edge $\eta$ appears in
face $i$}.
\end{aligned}
\end{equation}
Thus $a_{i\eta} = 0, 1,$ or $2$, and the sum of 
entries in each column is 
always $2$. The $\Gamma$ contribution of the space
$\pi^{-1}(p_1,\dots,p_n) = RG_{g,n}(\bp)$
 of metric ribbon graphs with 
a prescribed perimeter $\mathbf{p}=(p_1,\dots,p_n)$ is the orbifold 
polytope
$$
P_\Gamma (\mathbf{p})/\Aut(\Gamma),\qquad
P_\Gamma (\mathbf{p})= \{\mathbf{x}\in \bR_+ ^e\;|\;
A_\Gamma \mathbf{x} = \mathbf{p}\},
$$
where $\mathbf{x}=(\ell_1,\dots,\ell_e)$ is the collection of 
edge lengths of a metric ribbon graph $\Gamma$. We have
\begin{equation}
\label{eq:sump}
\sum_{i\in N} p_i= \sum_{i\in N}
\sum_{\eta\in E}a_{i\eta}\ell_\eta = 
2\sum_{\eta\in E}  
\ell_\eta.
\end{equation}

The canonical Euclidean 
volume $\vol(P_\Gamma(\mathbf{p}))$
of the polytope 
$P_\Gamma(\mathbf{p})$ is the ratio of the push-forward
measure of the Lebesgue measure on $\bR_+ ^e$ by $A_\Gamma$
and the Lebesgue measure on $\bR_+ ^n$ at the point 
$\mathbf{p}\in \bR_+ ^n$:
\begin{equation}
\label{eq:Euclideanvolume}
\vol(P_\Gamma(\mathbf{p}))
=\left. \frac{(A_\Gamma)_*
(d\ell_1\wedge \cdots\wedge d\ell_{e})}
{dp_1\wedge\cdots\wedge dp_n}\right|_\bp,
\end{equation}
where $(A_\Gamma)_*(d\ell_1\wedge \cdots\wedge 
d\ell_{e})$ is
the $n$-form on $\bR_+ ^n$ obtained by integrating 
the volume form on $\bR_+ ^{e}$ 
along the fiber $\pi^{-1}(\bp)$. This definition is equivalent to
imposing
\begin{equation}
\label{eq:volumeintegration}
\int_D \vol(P_\Gamma(\mathbf{p}))
{dp_1\wedge\cdots\wedge dp_n}
=\int_{A_\Gamma ^{-1}(D) }
d\ell_1\wedge \cdots\wedge d\ell_{e}
\end{equation}
for every open subset $D\subset \bR_+ ^n$ with compact
closure.
We define the \emph{Euclidean volume
 function} by
\begin{equation}
\label{eq:ve}
v_{g,n}^E(\bp)=
v_{g,n}^E(p_1,\dots,p_n) = 
\sum_{\substack{\Gamma {\text{ trivalent ribbon}}\\
{\text{graph of type }} (g,n)}} 
\frac{\vol(P_\Gamma(\mathbf{p}))}{|\Aut(\Gamma)|}.
\end{equation}
This is the Euclidean volume of the moduli space 
$\cM_{g,n}$ considered as the orbi-cell complex 
\begin{equation}
\label{eq:RGp}
RG_{g,n}(\bp)\overset{\text{def}}{=}
 \pi^{-1}(\bp)= 
\coprod_{\substack{\Gamma {\text{ ribbon graph}}\\
{\text{of type }} (g,n)}} \frac{P_\Gamma(\mathbf{p})}
{\Aut(\Gamma)}\isom \cM_{g,n}
\end{equation}
with the
prescribed perimeter length $\mathbf{p}\in \bR_+ ^n$.
Only degree $3$ (or trivalent)
graphs contribute to the volume function
because they parametrize the top dimensional cells.
Since $\dim_\bR RG_{g,n}(\bp) = 2(3g-3+n)$, 
we expect that the definition of
the push-forward measure and the relation
(\ref{eq:volumeintegration}) imply that the volume function
$v_{g,n}^E(\bp)$ has the polynomial growth of 
order $2(3g-3+n)$ as $\bp\rightarrow \infty$.
We will verify this growth order in Section~\ref{sect:Euc}, 
(\ref{eq:LTvolumeestimate2}).

\section{Topological recursion for the number of integral ribbon
graphs}
\label{sect:lattice}

It is a difficult task to find a topological recursion
formula for the Euclidean 
volume functions $v_{g,n}^E(\bp)$ directly from its definition. One might think that the Weil-Petersson
volume of the moduli of bordered hyperbolic 
surfaces \cite{Mir1, Mir2} would give the
Euclidean volume at the long boundary limit, but
actually the limit naturally converges to the
\emph{symplectic volume} we consider in
Section~\ref{sect:symp}. The straightforward
method for the Euclidean volume
is indeed to go through the detour of
considering the lattice point counting.
We therefore first derive a recursion formula for 
the number of metric ribbon graphs with integer 
edge lengths,  take
its Laplace transform, and then extract the topological 
recursion for the Euclidean volume functions.

Thus our main subject of this section 
is the set of all metric ribbon graphs $RG_{g,n} ^{\bZ_+}$
whose edges have
integer lengths. We call such a ribbon graph
an \emph{integral ribbon graph}.
Following \cite{N1},
let us define
 the weighted number $\big| RG_{g,n} ^{\bZ_+}(\bp)\big|$
of integral ribbon graphs with 
prescribed perimeter lengths
$\bp\in\bZ_+ ^n$:
\begin{equation}
\label{eq:Ngn}
N_{g,n}(\bp) = 
\big| RG_{g,n} ^{\bZ_+}(\bp)\big|
=\sum_{\substack{\Gamma {\text{ ribbon graph}}\\
{\text{of type }} (g,n)}}
\frac{\big|\{\bx\in \bZ_+ ^{e(\Gamma)}\;|\;A_\Gamma \bx = \bp\}
\big|}{|\Aut(\Gamma)|}.
\end{equation}
Since the finite set 
$\{\bx\in \bZ_+ ^{e(\Gamma)}\;|\;A_\Gamma \bx = \bp\}$
is a collection of lattice points in the polytope $P_\Gamma(\bp)$
with respect to the canonical integral structure $\bZ\subset
\bR$ of the real numbers, $N_{g,n}(\bp)$ can be thought of
as counting the 
number of \emph{lattice points}
in $RG_{g,n}(\bp)$ with a weight factor 
$1/|\Aut(\Gamma)|$ for each ribbon graph.
The function $N_{g,n}(\bp)$ is a symmetric function in
$\bp = (p_1,\dots,p_n)$
because the summation runs over all ribbon graphs of topological
type $(g,n)$ whose boundaries are labeled by the index set $N$.

\begin{rem}
The function (\ref{eq:Ngn}) was first considered in
\cite{N1}. Note that we do not allow the integer vector
$\bp\in\bZ_+^n$ to have any $0$ entry, since each face
of a ribbon graph must have a positive perimeter length. 
Note that $A_\Gamma\bx=0$ has no positive solutions.
Therefore,
the natural extension of the definition  (\ref{eq:Ngn})
to the case of $\bp = 0$ would
give $N_{g,n}(0)=0$. 
\end{rem}

Using the lattice point interpretation, it is easy to see
that the relation between this
function and the Euclidean volume function is the same
as that of the Riemann sum and the Riemann integral.
Let $k$ be a positive integer and $D\subset \bR_+ ^n$ an
open domain with compact closure. Then for every 
continuous function $f(\bp)$ on $D$, the definition of 
the Riemann integration in terms of Riemann sums gives
\begin{equation}
\label{eq:N=v}
\lim_{k\rightarrow \infty} \sum_{\bp\in D\cap \frac{1}{k}
\bZ_+ ^n}N_{g,n}(k\bp) f(\bp)
\frac{1}{k^{3(2g-2+n)}} = 
\int_D v_{g,n}^E(\bp) f(\bp) dp_1\cdots dp_n.
\end{equation}
This equality holds because our definition of the volume
uses the push-forward measure.
We note that as a function in $\bp$ there is no simple
direct relation between the values
$N_{g,n}(\bp)$ and $v_{g.n}^E(\bp)$. For example, 
 $N_{g,n}(\bp)=0$ if 
$\sum_{i=1} ^n p_i$ is odd because of (\ref{eq:sump}), but the
volume function is not subject to such a relation.

To derive a topological recursion for 
$N_{g,n}(\bp)$, we introduce the notion of 
\emph{ciliation}.

\begin{Def}
\label{def:ciliation}
A \emph{ciliation}
 is an assignment of a cilium in a face attached to a
bordering edge. Let $\ell\in\bZ_+$
be the length of the edge
 on which the ciliation is attached. We place the
root of the cilium at a half-integer length away from the 
vertices bounding the edge. Thus no cilium is attached to 
a vertex of a ribbon graph.
\end{Def}

\begin{figure}[htb]
\centerline{\epsfig{file=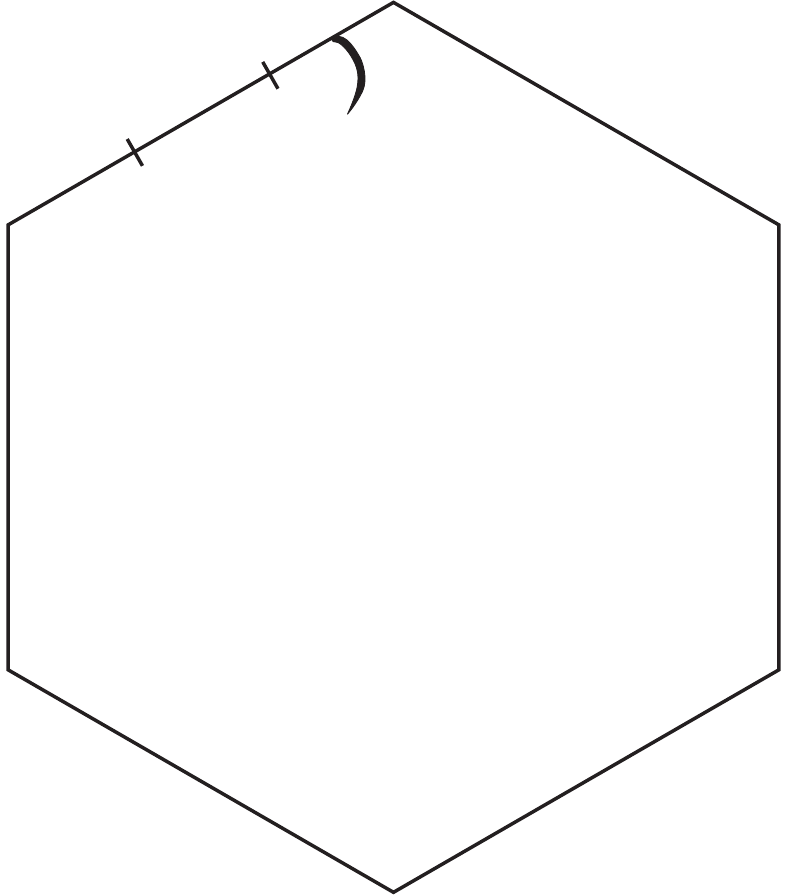, width=1in}}
\caption{A ciliation in a face. The cilium is placed
on a bordering edge, $0.5$ unit length away from
the nearest vertex.}
\label{fig:ciliation}
\end{figure}

The number of ciliations of a metric ribbon graph 
$\Gamma$ with integer edge lengths is given by
(\ref{eq:sump}). Indeed, if we count 
with respect to the edges, then there
are $2\ell$ ways for a ciliation to each edge 
because the cilium can be
placed on each side of the edge. And each face $i$ has
$p_i$ ways of ciliation. Thus the total number of
ciliations is $p_1+\cdots + p_n$.

For brevity of notation, we denote by $p_I = (p_i)_{i\in I}$
for a subset $I\in N=\{1,2\dots,n\}$. The cardinality of $I$ is
denoted by $|I|$.

\begin{thm}
\label{thm:integralrecursion}
The number of integral metric ribbon graphs
with prescribed boundary lengths satisfies the following 
topological recursion formula:
\begin{multline}
\label{eq:integralrecursion}
p_1  N_{g,n}(p_N)
=
\half
\sum_{j=2} ^n 
\Bigg[
\sum_{q=0} ^{p_1+p_j}
q(p_1+p_j-q)  N_{g,n-1}(q,p_{N\setminus\{1,j\}})
\\
+
H(p_1-p_j)\sum_{q=0} ^{p_1-p_j}
q(p_1-p_j-q)N_{g,n-1}(q,p_{N\setminus\{1,j\}})
\\
-
H(p_j-p_1)\sum_{q=0} ^{p_j-p_1}
q(p_j-p_1-q)
N_{g,n-1}(q,p_{N\setminus\{1,j\}})
\Bigg]
\\
+\half \sum_{0\le q_1+q_2\le p_1}q_1q_2(p_1-q_1-q_2)
\Bigg[
N_{g-1,n+1}(q_1,q_2,p_{N\setminus\{1\}})
\\
+\sum_{\substack{g_1+g_2=g\\
I\sqcup J=N\setminus\{1\}}} ^{\rm{stable}}
N_{g_1,|I|+1}(q_1,p_I)
N_{g_2,|J|+1}(q_2,p_J)\Bigg].
\end{multline}
Here 
$$
H(x) = \begin{cases}
1 \qquad x>0\\
0 \qquad x\le 0
\end{cases}
$$
is the Heaviside function, 
and the last sum is taken for all partitions
$g=g_1+g_2$ and $I\sqcup J=N\setminus \{1\}$
 subject to the stability condition
$2g_1-1+{I}>0$ and $2g_2-1+|J|>0$.
\end{thm}

\begin{proof}
The key idea is to count all integral ribbon graphs
with a cilium placed on the face named $1$. The number is 
clearly equal to $p_1  N_{g,n}(p_N)$. We then
analyze what happens when we remove the ciliated edge
from the ribbon graph. There are several situations after
the removal of this edge. The right-hand side of the 
recursion formula is obtained by the case-by-case analysis
of the edge removal operation. For any ciliated ribbon graph 
of type $(g,n)$ subject to 
the condition $2g-2+n >1$, removing the ciliated edge
creates a new graph of type
$(g, n-1)$ or $(g-1, n+1)$, or two disjoint graphs of types
$(g_1, n_1+1)$ and $(g_2,n_2+1)$ subject to the
stability condition and the partition condition
$$
\begin{cases}
g_1+g_2=g\\
n_1+n_2=n-1 .
\end{cases}
$$
Note that in each case the quantity $2g-2+n$ is reduced
exactly by $1$.

Let $\eta$ be the edge bordering face $1$ 
of a ribbon graph $\Gamma$ on which 
the cilium is placed, and $a_{1\eta}$
the incidence number of (\ref{eq:incidence}). 
Let $\ell\in \bZ_+$ be the length of edge $\eta$.
There are two main situations: 
$a_{1\eta} = 1$ and $a_{1\eta} = 2$. Each main situation 
breaks down further into three cases.
Before examining each care in detail, we
first need to analyze the effect of $\Aut(\Gamma)$
in the edge removal operation. Note that
the automorphism group fixes each face. Thus
$\eta$ moves to another edge $\eta'$ of face $1$. 
If $\eta = \eta'$, then the automorphism
is unaffected by the edge removal and we have
$\Aut(\Gamma) = \Aut(\Gamma\setminus \eta)$,
where the right-hand side is a product group
if $\Gamma\setminus \eta$ is disconnected. 
If $\eta\ne \eta'$, then placing a cilium on $\eta$
or $\eta'$ inside face $1$ is indistinguishable,
 and this identification is
accounted for in the counting $p_1 N_{g,n}(p_N)$.

\begin{case}
$a_{1\eta} = a_{j\eta}=1$ for  $j\ge 2$,
 $p_1>\ell$ and $p_j>\ell$.
Define 
$q = (p_1-\ell) + (p_j-\ell) >0.$ Then we have
$$
\begin{cases}
q-p_1+p_j = 2(p_j-\ell)>0\\
q+p_1-p_j=2(p_1-\ell)>0.
\end{cases}
$$
Therefore, $q>|p_1-p_j|$. Geometrically, $q$ is the perimeter
length
of the face created by removing edge $\eta$ that separates
faces $1$ and $j$ (see Figure~\ref{fig:case1}). 

To recover the
original ribbon graph with a cilium on edge $\eta$ 
of length $\ell$ from
the one without edge $\eta$, we need to place the edge  
 on the face of perimeter $q$, and place 
a cilium on this edge. Here we note that the data
$p_i, p_j, q$ and $\ell$ are all prescribed.
The number of ways to place an endpoint of
the edge on the face of perimeter length $q$ is $q$.  
This point uniquely determines the edge
we need, since the other endpoint is $p_1-\ell$ away from
the first endpoint along the perimeter 
measured by the clockwise distance. 
The enclosed face of perimeter length $p_1$ becomes face $1$,
and the other side of the newly placed edge is face $j$.
Since the ciliation is done on face $1$, there are $\ell$
choices for the assignment of the root of the 
cilium. Altogether, the contribution of 
this case is
\begin{equation}
\label{eq:case1}
\sum_{q=|p_1-p_j|+1} ^{p_1+p_j}
q\;\frac{p_1+p_j-q}{2} \; N_{g,n-1}(q,p_{N\setminus\{1,j\}}).
\end{equation}
\end{case}

\begin{figure}[htb]
\centerline{\epsfig{file=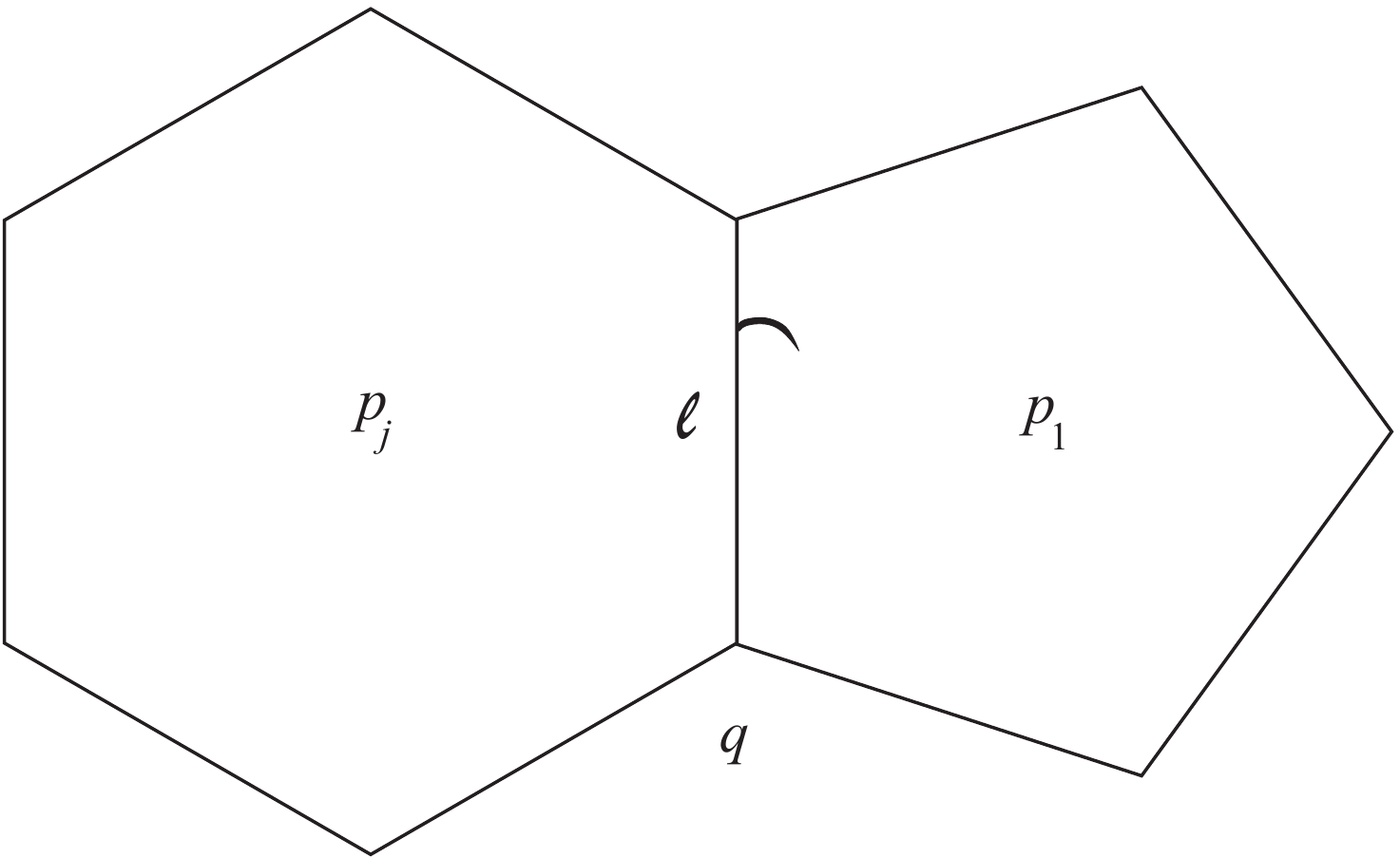, width=2in}}
\caption{Case $1$: $a_{1\eta} = a_{j\eta}=1$,
 $p_1>\ell$ and $p_j>\ell$.}
\label{fig:case1}
\end{figure}

\begin{case} $a_{1\eta} = a_{j\eta}=1$ for  $j\ge 2$, and
 $p_1\ge p_j =\ell$. Since $p_j=\ell$, face $j$ and
 edge $\eta$ are the same and forms a loop. This loop is connected
 to face $1$ by an edge $\eta'$ of incidence number $2$. Let
 $\ell'$ be the length of this connecting edge, which is bounded
 by $(p_1-p_j)/2 \ge \ell'\ge 0$ (see Figure~\ref{fig:case2}, left).
 This time define $q=p_1-p_j-2\ell'$. This is the perimeter length
 of the face created by removing face $j$ and edge $\eta'$.
 In this situation, removing edge $\eta$ ($=$ face $j$)
 alone does not create
 an admissible ribbon graph, since  edge $\eta'$ remains
 with a vertex of degree $1$ at one end. Therefore, we need
 to remove the entire \emph{tadpole} consisting of a
 head of face $j$ and a tail of edge $\eta'$. The cilium is 
 on face $1$, which is attached to the outer boundary of face
 $j$.
 
 To recover the original graph from the result of this tadpole
 removal, we have $q$ choices for the tadpole placement
 and $p_j=\ell$ choices for ciliation. Therefore, the contribution
 from this case is 
\begin{equation}
\label{eq:case2}
\sum_{q=0} ^{p_1-p_j}
qp_j N_{g,n-1}(q,p_{N\setminus\{1,j\}}).
\end{equation}
\end{case}

\begin{figure}[htb]
\centerline{\epsfig{file=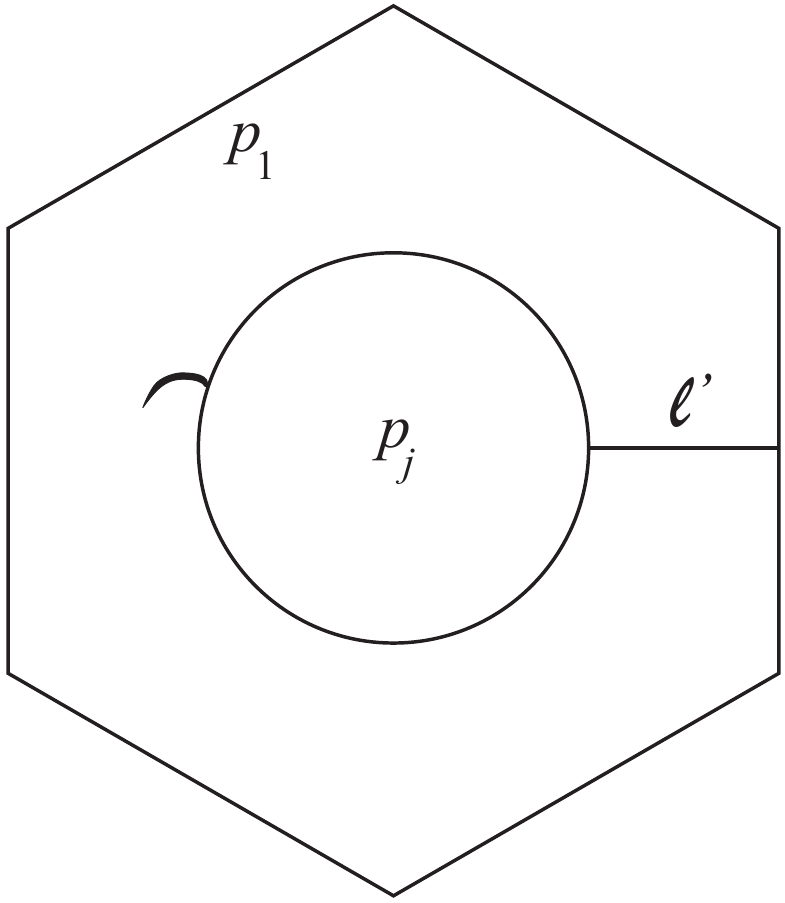, width=1.2in}
\hskip0.3in
\epsfig{file=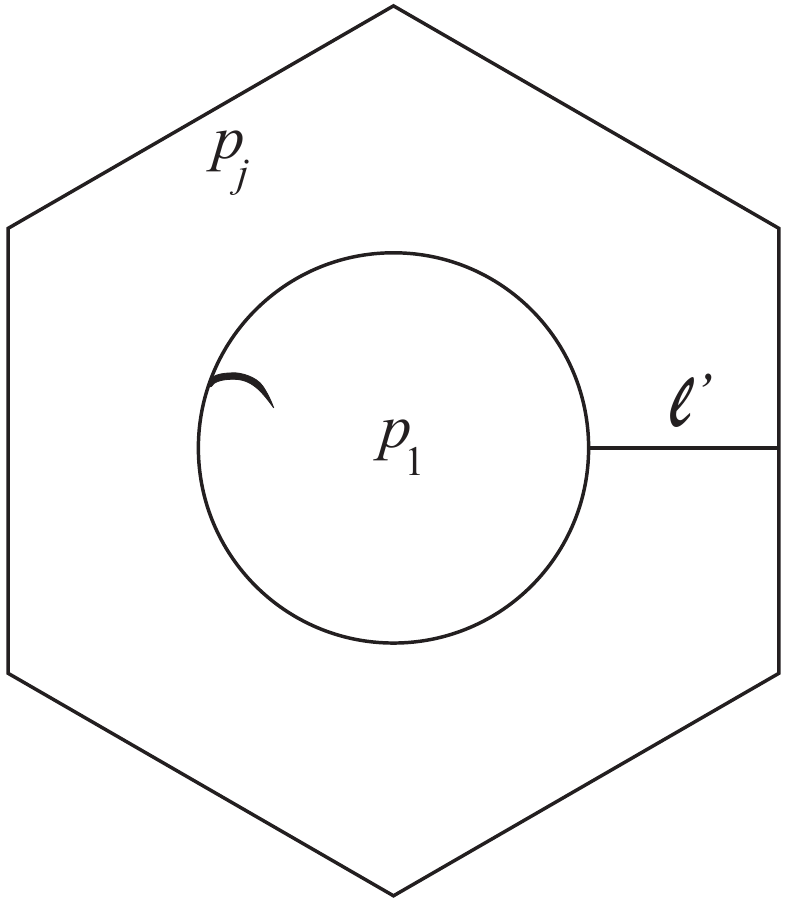, width=1.2in}
\hskip0.3in
\epsfig{file=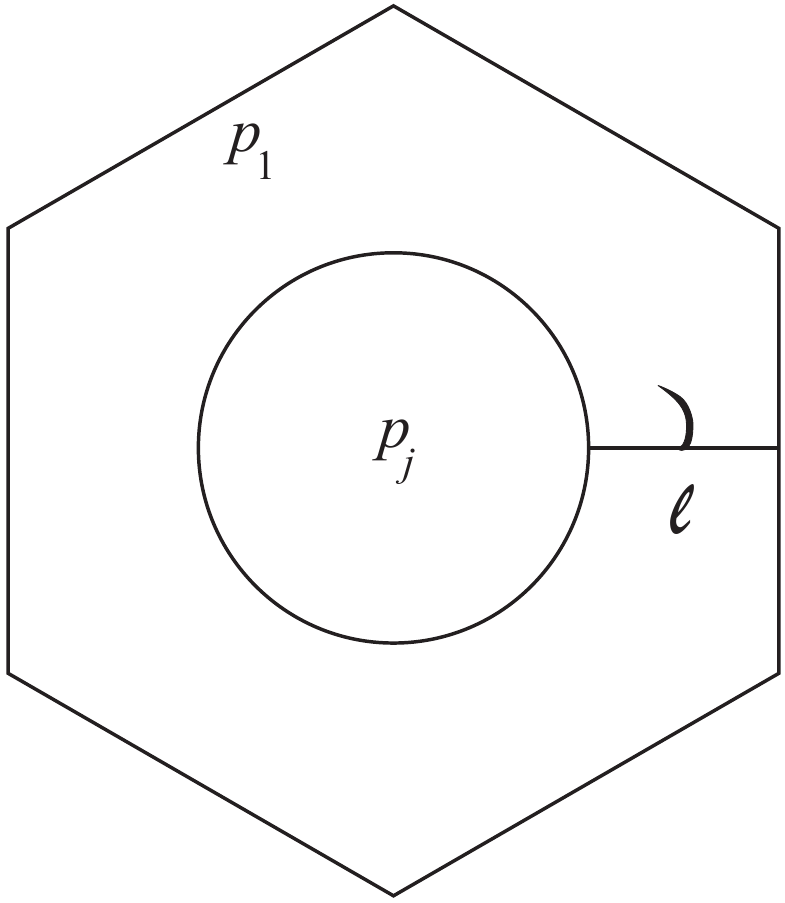, width=1.2in}}
\caption{Case $2$ (left): $a_{1\eta} = a_{j\eta}=1$ and
 $p_1\ge p_j=\ell$; Case $3$ (center):
 $a_{1\eta} = a_{j\eta}=1$ and
 $p_j\ge p_1=\ell$; and Case $4$ (right):
 $a_{1\eta}=2$ and the edge $\eta$ connects a loop
 $j$ to the rest of the graph.}
\label{fig:case2}
\end{figure}

\begin{case}
$a_{1\eta} = a_{j\eta}=1$ for  $j\ge 2$, and
 $p_j\ge p_1 =\ell$. The situation is similar to 
 Case~2 (see Figure~\ref{fig:case2}, center).
 Let $\eta'$ be the edge of length $\ell'$ that connects face 
 $1$ and face $j$.
 Define $q=p_j-p_1-2\ell'$. This is the perimeter length
 of the face created by removing the entire tadpole consisting
 of face $1$ with a cilium as its head and edge $\eta'$ as its tail.
 We have $q$ choices for tadpole placement and $p_1$ choices
 for ciliation. 
 Thus the contribution is 
 \begin{equation}
\label{eq:case3}
\sum_{q=0} ^{p_j-p_1}
qp_1 N_{g,n-1}(q,p_{N\setminus\{1,j\}}).
\end{equation}
\end{case}

\begin{case} $a_{1\eta}=2$ and removal of edge $\eta$
separates a single loop $j$ for some $j\ge 2$ from the rest of the
graph (see Figure~\ref{fig:case2}, right). It is necessary 
that $p_1>p_j$ in this case.
Since a single loop
alone is not an admissible graph, we need to remove
face $j$ together when we remove edge $\eta$. 
Define $q=p_1-p_j-2\ell$, which is the perimeter length of the
face created after the removal of the tadpole. This time 
the recovery process has $q$ choices of tadpole placement
and $2\ell$ choices for ciliation, because the cilium can be
placed on either side of the tail. Thus the contribution is
 \begin{equation}
\label{eq:case4}
\sum_{q=0} ^{p_1-p_j}
q(p_1-p_j-q) N_{g,n-1}(q,p_{N\setminus\{1,j\}}).
\end{equation}
\end{case}

\begin{case} $a_{1\eta}=2$ and removal of edge $\eta$
creates a connected ribbon graph. 
The removal of edge $\eta$ breaks face $1$ into two separate
faces of perimeter lengths $q_1$ and $q_2$ subject to the
condition $0<q_1+q_2<p_1$.
The removal of
the edge reduces the genus by $1$, and increases the number of 
faces by $1$. We have the equality
$p_1 = q_1+q_2+2\ell$ (see Figure~\ref{fig:case5}).
To recover the original graph from the result of the edge removal,
we have $q_1$ choices for one endpoint of edge $\eta$, 
$q_2$ choices for the other endpoint, and $2\ell$ choices for
ciliation, again because the cilium can be placed on either side of 
edge $\eta$. Altogether the contribution is
 \begin{equation}
\label{eq:case5}
\half \sum_{0\le q_1+q_2\le p_1}
q_1q_2(p_1-q_1-q_2) N_{g-1,n+1}(q_1,q_2,p_{N\setminus\{1\}}).
\end{equation}
Here we need the factor $\half$, which is the symmetry factor
of interchanging $q_1$ and $q_2$. 
\end{case}

\begin{figure}[htb]
\centerline{\epsfig{file=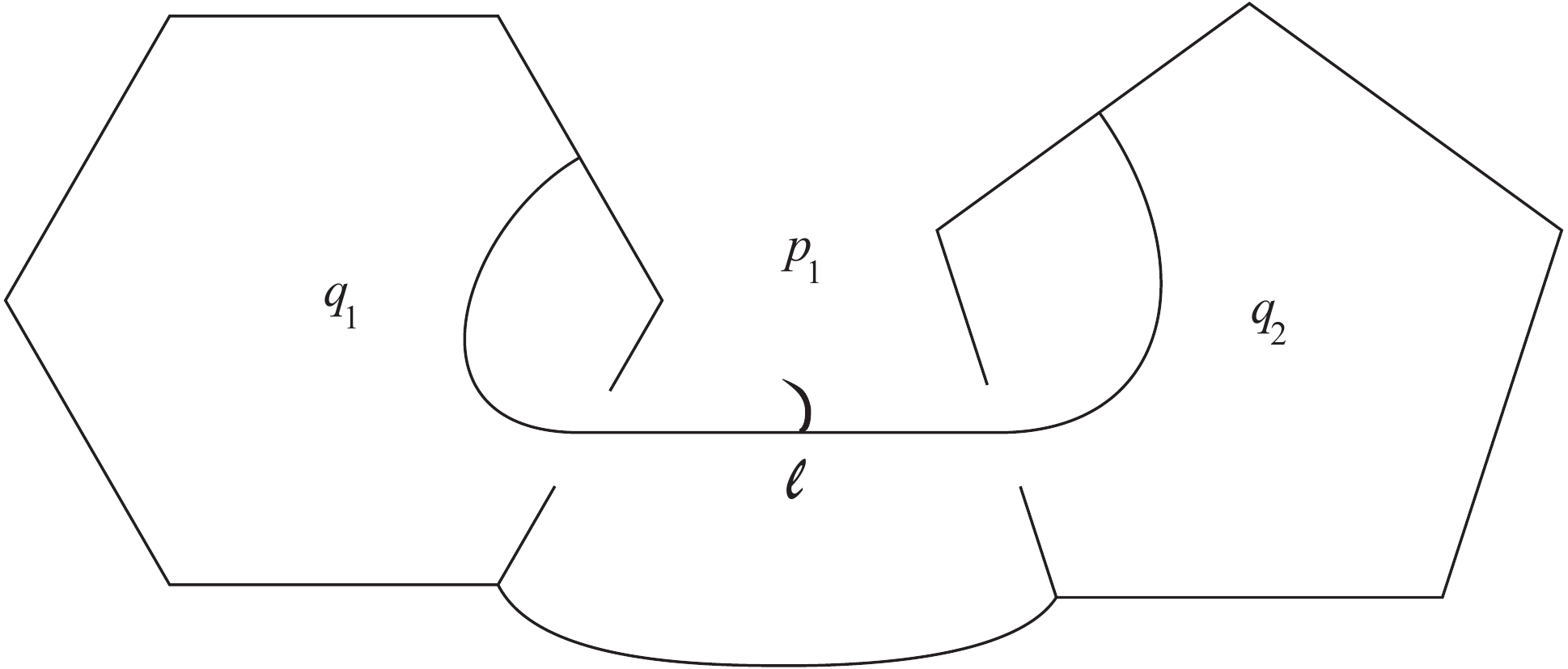, width=2.5in}}
\caption{Case $5$: $a_{1\eta}=2$
and removal of edge $\eta$
creates a connected ribbon graph.}
\label{fig:case5}
\end{figure}

\begin{case} $a_{1\eta}=2$ and removal of edge $\eta$
creates a disjoint union of two ribbon graphs. 
There are $n$ faces in the original ribbon graph $\Gamma$.
The removal of edge $\eta$ breaks face $1$ into two separate
faces of perimeter lengths $q_1$ and $q_2$. The other
faces $2, 3, \dots, n$ remain intact. Let $I\subset N\setminus\{1\}$
be the label of faces that are connected to the new face of perimeter
length $q_1$, and $J\subset N\setminus\{1\}$ for $q_2$. 
Then the two disjoint ribbon graphs have types $(g_1,|I|+1)$
and $(g_2,|J|+1)$ satisfying the partition condition 
$$
\begin{cases}
g_1+g_2=g\\
I\sqcup J=N\setminus\{1\}.
\end{cases}
$$
The contribution from this case is
\begin{equation}
\label{eq:case6}
\half \sum_{0\le q_1+q_2\le p_1}
q_1q_2(p_1-q_1-q_2)
\sum_{\substack{g_1+g_2=g\\
I\sqcup J=N\setminus\{1\}}} ^{\rm{stable}}
N_{g_1,|I|+1}(q_1,p_I)
N_{g_2,|J|+1}(q_2,p_J)
\end{equation}
with the symmetry factor $\half$ corresponding to 
interchanging $q_1$ and $q_2$.
\end{case}

\begin{figure}[htb]
\centerline{\epsfig{file=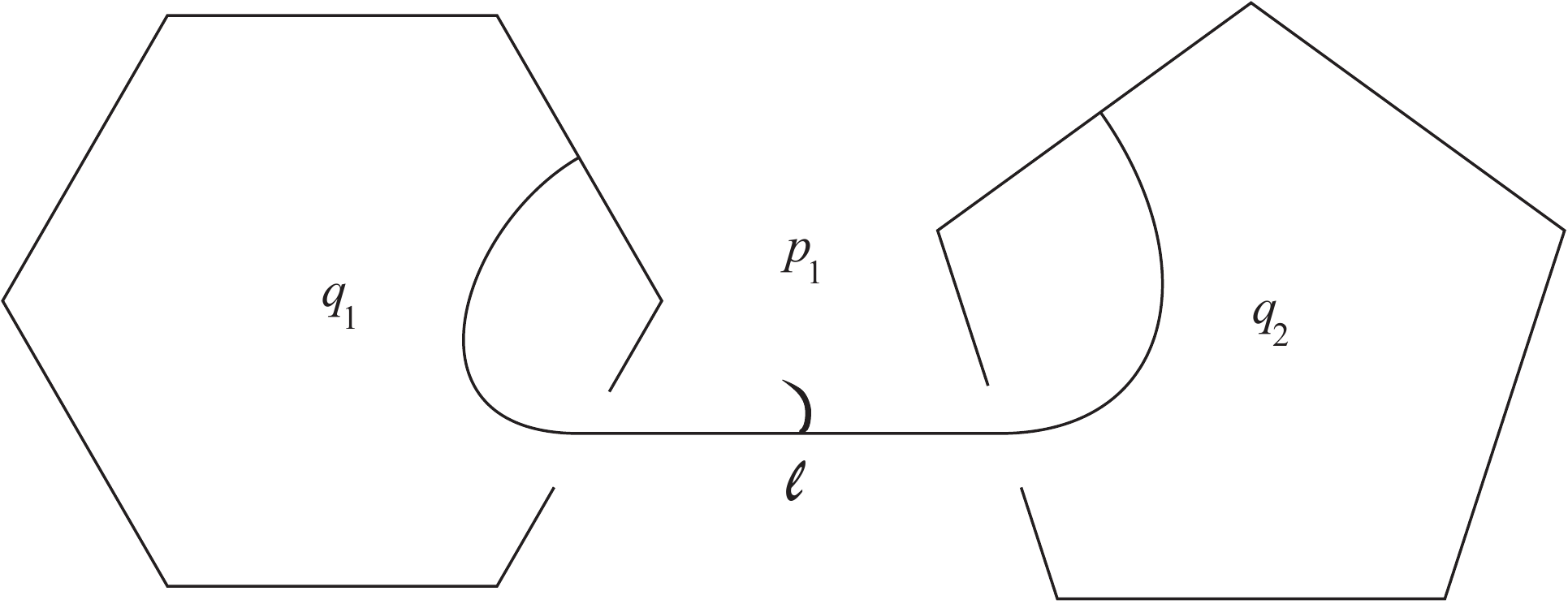, width=2.5in}}
\caption{Case $6$: $a_{1\eta}=2$
and removal of edge $\eta$
creates a disjoint union of two ribbon graphs.}
\label{fig:case6}
\end{figure}

Summing all contributions (\ref{eq:case1})-(\ref{eq:case6}),
 we obtain
\begin{multline}
\label{eq:integralrecursion2}
p_1  N_{g,n}(p_N)
=
\sum_{j=2} ^n \sum_{q=|p_1-p_j|+1} ^{p_1+p_j}
q\;\frac{p_1+p_j-q}{2} \; N_{g,n-1}(q,p_{N\setminus\{1,j\}})
\\
+
\sum_{j=2} ^n 
H(p_1-p_j)\sum_{q=0} ^{p_1-p_j}
q(p_1-q) \; N_{g,n-1}(q,p_{N\setminus\{1,j\}})
\\
+
\sum_{j=2} ^n
H(p_j-p_1)\sum_{q=0} ^{p_j-p_1}
q p_1 \;
N_{g,n-1}(q,p_{N\setminus\{1,j\}})
\\
+\half \sum_{0\le q_1+q_2\le p_1}q_1q_2(p_1-q_1-q_2)
\Bigg[
N_{g-1,n+1}(q_1,q_2,p_{N\setminus\{1\}})
\\
+\sum_{\substack{g_1+g_2=g\\
I\sqcup J=N\setminus\{1\}}} ^{\rm{stable}}
N_{g_1,|I|+1}(q_1,p_I)
N_{g_2,|J|+1}(q_2,p_J)\Bigg].
\end{multline}
If we allow the variable $q$ to range from $0$ to 
$p_1+p_j$ in the first summation
of the right-hand side of (\ref{eq:integralrecursion2}), 
then we need to compensate the non-existing
cases. Note that  we have
\begin{multline}
\label{eq:intermediate}
-\sum_{j=2} ^n \sum_{q=0} ^{|p_1-p_j|}
q\;\frac{p_1+p_j-q}{2} \; N_{g,n-1}(q,p_{N\setminus\{1,j\}})
\\
+
\sum_{j=2} ^n 
H(p_1-p_j)\sum_{q=0} ^{p_1-p_j}
q\; \frac{2p_1-2q}{2} \; N_{g,n-1}(q,p_{N\setminus\{1,j\}})
\\
+
\sum_{j=2} ^n
H(p_j-p_1)\sum_{q=0} ^{p_j-p_1}
q \; \frac{2p_1}{2} \;
N_{g,n-1}(q,p_{N\setminus\{1,j\}})
\\
=
\sum_{j=2} ^n 
H(p_1-p_j)\sum_{q=0} ^{p_1-p_j}
q\; \frac{p_1-p_j-q}{2} \; N_{g,n-1}(q,p_{N\setminus\{1,j\}})
\\
-
\sum_{j=2} ^n
H(p_j-p_1)\sum_{q=0} ^{p_j-p_1}
q \; \frac{p_j-p_1-q}{2} \;
N_{g,n-1}(q,p_{N\setminus\{1,j\}}).
\end{multline}
Substituting (\ref{eq:intermediate}) in (\ref{eq:integralrecursion2}),
we obtain (\ref{eq:integralrecursion}).
This completes the proof.
\end{proof}

\begin{rem} The topological recursion
 for 
$N_{g,n}(\bp)$ was first considered  by Norbury in \cite{N1}. 
His proof is similar in that it involved an edge
removal operation, but
the main formula and its proof therein  contained
are incorrectly recorded -- the terms involving products of functions
$N_{g,n}$ were double counted and need a compensating factor of
$\frac{1}{2}$. A corrected version appears in \cite{DN1, N3}.
Our proof presented
 here is new, and is
 based on  a different idea using ciliation.
\end{rem}

\section{The Laplace transform of the number of 
integral ribbon graphs}
\label{sect:LTintegral}

The limit formula (\ref{eq:N=v}) tells us that
$N_{g,n}(\bp)$ 
asymptotically behaves like a polynomial
for  large $\bp\in\bZ_+ ^n$, and the coefficients of the leading
terms correspond to that of
 the Euclidean volume function $v_{g,n}^E (\bp)$. 
 The lack of the direct relation between 
 $N_{g,n}(\bp)$ and $v_{g,n}^E (\bp)$, together with 
 equation  (\ref{eq:N=v}), suggest that we need to 
 consider an integral transform, such as 
 the Laplace transform of $N_{g,n}(\bp)$, to extract the 
 information of the Euclidean volume of $RG_{g,n}(\bp)$
 from it.
Since 
  $$
 \int_0 ^\infty x^m e^{-xw}dx = \frac{m!}{w^{m+1}}
 $$
 for a complex variable $w\in \bC$ with $Re(w)>0$,
the coefficients of the highest order poles of
 the Laplace transform
 \begin{equation}
 \label{eq:Lgn}
 L_{g,n}(w_1,\dots,w_n) \overset{\rm{def}}{=} 
 \sum_{\bp\in\bZ_{+} ^n} N_{g,n}(\bp) e^{-\la \bp,w\ra}
 \end{equation}
 should
represent the Euclidean volume of $RG_{g,n}(\bp)$.
Here $\la p,w\ra=p_1w_1+\cdots+p_nw_n$. 
This section is devoted to the analysis of the Laplace
transform of the topological recursion
(\ref{eq:integralrecursion}).

To relate our investigation with the Hurwitz theory
and the Witten-Kontsevich theory, and  in particular from 
the point of view of the \emph{polynomial} expressions of
\cite{EMS, MZ}, we introduce new complex coordinates
\begin{equation}
\label{eq:t}
e^{-w} = \frac{t+1}{t-1}\qquad
{\rm{and}}\qquad e^{-w_j} = \frac{t_j+1}{t_j-1},
\end{equation}
and express the result of the Laplace transform in terms of these
$t$-variables. 
This substitution makes sense because the Laplace transform
is a rational function in $e^{-w_j}$'s.

\begin{thm}
\label{thm:LTrecursion} 
Define 
$\cL_{g,n}(t_1,\dots,t_n)$ by
\begin{multline}
\label{eq:calL}
\cL_{g,n}(t_1,\dots,t_n) \;dt_1\tensor\cdots
\tensor dt_n
=
(d_1\tensor\cdots\tensor d_n)
L_{g,n}\big(w_1(t),\dots,w_n(t)\big)
\\
=
 \frac{\partial^n}{\partial w_1\cdots\partial w_n}
L_{g,n}(w_1,\dots,w_n)\;dw_1\tensor\cdots\tensor dw_n
\end{multline}
using the coordinate change {\rm{(\ref{eq:t})}}.
The differentials $dt_j$ and $dw_j$ are related by
$$
dw_j = \frac{2}{t_j ^2-1} \; dt_j.
$$
Then every
$\cL_{g,n}(t_1,\dots,t_n) $
for $2g-2+n>0$ is a
Laurent polynomial of degree $3g-3+n$ in $t_1 ^2, t_2 ^2,\dots,
t_n ^2$. The initial values are
\begin{equation}
\label{eq:LT03}
\cL_{0,3}(t_1,t_2,t_3) = -\frac{1}{16}\left(
1-\frac{1}{t_1 ^2\; t_2 ^2 \;t_3 ^2}
\right)
\end{equation}
and 
\begin{equation}
\label{eq:LT11}
\cL_{1,1}(t)=-\frac{1}{128}\cdot \frac{(t^2-1)^3}{t^4}.
\end{equation}
 The functions 
 $\cL_{g,n}(t_1,\dots,t_n) $ for all $(g,n)$ subject to
 $2g-2+n>0$ are uniquely
  determined by the topological recursion formula
\begin{multline}
\label{eq:LTrecursion}
\cL_{g,n}(t_N)
\\
=
-\frac{1}{16}\sum_{j=2} ^n 
\frac{\partial}{\partial t_j}
\left[
\frac{t_j}{t_1^2-t_j^2}
\left(
\frac{(t_1^2-1)^3}{t_1^2}
\cL_{g,n-1}(t_{N\setminus\{j\}})
-\frac{(t_j^2-1)^3}{t_j^2}
\cL_{g,n-1}(t_{N\setminus\{1\}})
\right)
\right]
\\
-\frac{1}{32}\; \frac{(t_1^2-1)^3}{t_1^2}
\left[
\cL_{g-1,n+1}(t_1,t_1,t_{N\setminus\{1\}})
+\sum_{\substack{g_1+g_2=g\\I\sqcup J=N\setminus\{1\}}}
^{\rm{stable}}
\cL_{g_1,|I|+1}(t_1,t_I)\cL_{g_2,|J|+1}(t_1,t_J)
\right].
\end{multline}
Here we use the same convention of notations as in 
Theorem~{\rm{\ref{thm:integralrecursion}}}. 
\end{thm}

\noindent
If we assume (\ref{eq:LT03}), (\ref{eq:LT11}), and
(\ref{eq:LTrecursion}), then it is obvious that
$\cL_{g,n}(t_N)$ is a
Laurent polynomial in $t_1^2, \dots, t_n^2$ of degree
$3g-3+n$. 
The proof of (\ref{eq:LTrecursion}) is given in 
Appendix~\ref{app:LTProof}.
The initial values (\ref{eq:LT03})
and  (\ref{eq:LT11}) are calculated in
Appendix~\ref{app:examples}.

\section{The Euclidean volume of the moduli space}
\label{sect:Euc}

In this section we extract the information on 
the Euclidean volume function
from the Laurent polynomial $\cL_{g,n}(t_N)$. We then derive
a topological recursion for the Laplace transform of 
the Euclidean volume.
Let us recall the  Euclidean volume
function $v_{g,n}^E(\bp)$
of {\rm{(\ref{eq:ve})}}.

\begin{prop}
\label{prop:VE}
Let $V_{g,n}^E(t_N)$ be
the homogeneous leading terms of $\cL_{g,n}(t_N)$
for $(g,n)$ subject to $2g-2+n>0$.
Then we have
\begin{equation}
\label{eq:VE}
V_{g,n}^E(t_N) dt_1\tensor \cdots\tensor dt_n
= 
 d_1\tensor \cdots\tensor d_n 
 \int_{\bR_+^n} 
v_{g,n}^E(\bp)
e^{-\la w,\bp\ra}dp_1\cdots dp_n,
\end{equation}
where we change the $w$-variables to the $t$-variables according 
to the transformation 
{\rm{(\ref{eq:t})}}.
\end{prop}

\begin{proof}
From (\ref{eq:N=v}), we have
\begin{multline*}
\int_{\bR_+^n} 
v_{g,n}^E(\bp) 
e^{-\la w,\bp\ra}dp_1\cdots dp_n
=
\lim_{k\rightarrow \infty}
\sum_{\bp\in \frac{1}{k}\bZ_+ ^n}
N_{g,n}(k\bp) 
e^{-\la w,\bp\ra} \frac{1}{k^{3(2g-2+n)}}
\\
=
\lim_{k\rightarrow \infty}
\sum_{\bp\in \bZ_+ ^n}
N_{g,n}(\bp) 
e^{-\frac{1}{k}\la w,\bp\ra} \frac{1}{k^{3(2g-2+n)}}
\\
=
\lim_{k\rightarrow \infty}
L_{g,n}\left(\frac{w_1}{k},\cdots,\frac{w_n}{k}\right) \frac{1}{k^{3(2g-2+n)}}.
\end{multline*}
The coordinate transformation (\ref{eq:t}) has the expansion
near $w=0$
\begin{equation}
\label{eq:wandt}
\begin{aligned}
t&=t(w)=-\frac{2}{w}-\frac{w}{6}+\frac{w^3}{360}-\frac{w^5}{15120}
+\cdots,
\\
w&=w(t)=-\frac{2}{t}-\frac{2}{3t^3}-\frac{2}{5t^5}-\cdots.
\end{aligned}
\end{equation}
Since
$$
\cL_{g,n}(t_N)=
\frac{\partial^n}{\partial t_1\cdots \partial t_n}
L_{g,n}\big( w(t_1),\dots,w(t_n)\big) 
$$
is a Laurent polynomial of degree $2(3g-3+n)$, and since
the change $w\mapsto w/k$ makes
$$
t\longmapsto k\; t+\cO\left(\frac{1}{k}\right)
$$
for a fixed value $t$, 
we have
\begin{multline*}
\frac{\partial^n}{\partial t_1\cdots \partial t_n}
\int_{\bR_+^n} 
v_{g,n}^E(\bp) 
e^{-\la w(t),\bp\ra}dp_1\cdots dp_n
\\
=
\lim_{k\rightarrow \infty}
\frac{\partial^n}{\partial t_1\cdots \partial t_n}\;
L_{g,n}\left(\frac{w(t_1)}{k},\cdots,\frac{w(t_n)}{k}\right)
\; \frac{1}{k^{3(2g-2+n)}}
\\
=
\lim_{k\rightarrow \infty}
\cL_{g,n}\left(
kt_1+\cO\left(\frac{1}{k}\right),\cdots,
kt_n+\cO\left(\frac{1}{k}\right)
\right)
\; \frac{k^n}{k^{3(2g-2+n)}}
=V_{g,n}^E(t_N).
\end{multline*}
This completes the proof.
\end{proof}

Since  $v_{g,n}^E(\bp)$ is defined by the push-forward
measure of the incidence matrix $A_\Gamma$
of (\ref{eq:incidence}) at each point $\Gamma\in RG_{g,n}(\bp)$,
we have
\begin{multline}
\label{eq:LTvolumeestimate2}
\int_{\bR_+^n} 
v_{g,n}^E(\bp)
e^{-\la w,\bp\ra}dp_1\cdots dp_n
\\
=
\sum_{\substack{\Gamma {\text{ trivalent ribbon}}\\
{\text{graph of type }} (g,n)}} 
\frac{1}{|\Aut(\Gamma)|}
\int_{\bR_+^n}
\vol(P_\Gamma(\mathbf{p}))
e^{-\la w,\bp\ra}dp_1\cdots dp_n
\\
=
\sum_{\substack{\Gamma {\text{ trivalent ribbon}}\\
{\text{graph of type }} (g,n)}} 
\frac{1}{|\Aut(\Gamma)|}
\int_{\bR_+^{e(\Gamma)}}
e^{-\la w,A_\Gamma \mathbf{x} \ra}dx_1\cdots dx_{e(\Gamma)}
\\
=
\sum_{\substack{\Gamma {\text{ trivalent ribbon}}\\
{\text{graph of type }} (g,n)}} 
\frac{1}{|\Aut(\Gamma)|}
\prod_{\eta=1} ^{e(\Gamma)} \frac{1}{\la w,a_\eta\ra},
\end{multline} 
where $a_1, \dots, a_{e(\Gamma)}$ are columns of the 
edge-face incidence matrix 
$$
A_\Gamma
=\big[a_1\big|a_2\big|\cdots\big |a_{e(\Gamma)}\big].
$$
We note that  $e(\Gamma)$ takes its maximum value
$3(2g-2+n)$ for a trivalent graph. Thus 
the last line of (\ref{eq:LTvolumeestimate2})
has a pole of order $3(2g-2+n)$ at $w=0$.
This expression also shows that the leading terms 
of $L_{g,n}\big(w(t_N)\big)$ 
as a function in $t_N$ using the expansion 
(\ref{eq:wandt}) around $t_N\sim\infty$
are the Laplace transform of the
Euclidean volume function. In particular, 
we deduce that $N_{g,n}(\bp)$ behaves asymptotically like
a polynomial of degree $2(3g-3+n)$ for  large $\bp\in\bR_+^n$.

Since $V_{g,n} ^E(t_N)$ is the leading terms of
$\cL_{g,n}(t_N)$, it is easy to obtain a topological 
recursion.

\begin{thm}
\label{thm:LTofErecursion}
The Laplace transformed Euclidean volume function
$V_{g,n} ^E(t_N)$ in the stable range $2g-2+n>0$
satisfies the following topological 
recursion:
\begin{multline}
\label{eq:LTErecursion}
V_{g,n} ^E(t_N)
=-\frac{1}{16}\sum_{j=2} ^n 
\frac{\partial}{\partial t_j}
\left[
\frac{t_j}{t_1^2-t_j^2}
\bigg(
t_1^4
V_{g,n-1}^E(t_{N\setminus\{j\}})
-t_j^4
V_{g,n-1}^E(t_{N\setminus\{1\}})
\bigg)
\right]
\\
-\frac{1}{32}\; t_1^4
\left[
V_{g-1,n+1}^E(t_1,t_1,t_{N\setminus\{1\}})
+\sum_{\substack{g_1+g_2=g\\I\sqcup J=N\setminus\{1\}}}
^{\rm{stable}}
V_{g_1,|I|+1}^E(t_1,t_I) V_{g_2,|J|+1}^E(t_1,t_J)
\right].
\end{multline}
\end{thm}

\begin{proof}
The leading contribution of (\ref{eq:LTrecursion})
comes from the leading term of
$$
\frac{(t^2-1)^3}{t^2} = t^4-3t^2+3-\frac{1}{t^2}.
$$
Thus  (\ref{eq:LTrecursion})
reduces to (\ref{eq:LTErecursion}).
\end{proof}

\section{The symplectic volume of the moduli space
and the Kontsevich constants}
\label{sect:symp}

Suppose the $i$-th face of a metric ribbon graph 
 $\Gamma\in RG_{g,n}(\bp)$ consists
of edges labeled by $1, 2, \dots, k$ in 
 this cyclic order. 
 (Here again we are abusing the notation to indicate a metric
 ribbon graph by the same letter $\Gamma$.)
 If an edge appears twice in this list, then we
 count it repetitively.
 Denote by $\ell_\a$ the length of edge $\a$. 
 They satisfy the relation $\ell_1+\cdots+\ell_k=p_i$. 
 Note that the collection of edge lengths forms an 
 orbifold coordinate system
 on $RG_{g,n}$ at each point $\Gamma$.
Kontsevich \cite{K1992} defines 
 a $2$-form on $RG_{g,n}$  by
 \begin{equation}
 \label{eq:omegap}
 \omega_K(\bp)=\sum_{i=1} ^n p_i ^2 \omega_i,\qquad
 \omega_i = \sum_{\a<\b}d\left(\frac{\ell_\a}{p_i}\right)
 \wedge d\left(\frac{\ell_\b}{p_i}\right) \quad{\text{on face }}
\; i.
 \end{equation}
 If we change the cyclic order from 
 $(1,2,\dots,k)$ to $(2,3,\dots,k,1)$
 and define the form $\omega_i '$ in the same manner, then
 we have 
 $$
 \omega_i-\omega_i ' = 2 d\left(\frac{\ell_1}{p_i}\right)
  \wedge d\left(\frac{\ell_2+\cdots+\ell_k}{p_i}\right) = 0.
 $$
 Therefore, each  $\omega_i$ and $\omega_K(\bp)$
 are well defined as  genuine $2$-forms on $RG_{g,n}$. 
 The restriction of the $2$-form $\omega_K(\bp)$ defines a symplectic
 structure on $RG_{g,n}(\bp)\isom \cM_{g,n}$ for each 
 $\bp\in\bR_+ ^n$.

 To see the non-degeneracy of $\omega_K(\bp)$, 
 let us analyze the perimeter map $\pi$ locally around
 a trivalent ribbon graph
 $\Gamma$. As in Section~\ref{sect:combinatorial}
  we give a name to all edges of $\Gamma$,
 this time without repetition, indexed by
 $\{0,1,2,\dots,e(\Gamma)-1\}$. 
 Faces of $\Gamma$ are indexed by $N=\{1,2,\dots,n\}$.
The edge-face incidence matrix $A_\Gamma$ of
(\ref{eq:incidence})
gives
 the differential of the perimeter map
 $$
 A_\Gamma=d\pi_\Gamma 
 $$
 at the metric ribbon graph $\Gamma$ if it is trivalent.
 To set notations simple, we assume that faces $1$ through $4$
 and edges $0$ through $4$ are arranged as in 
 Figure~\ref{fig:vfield}.

\begin{figure}[htb]
\centerline{\epsfig{file=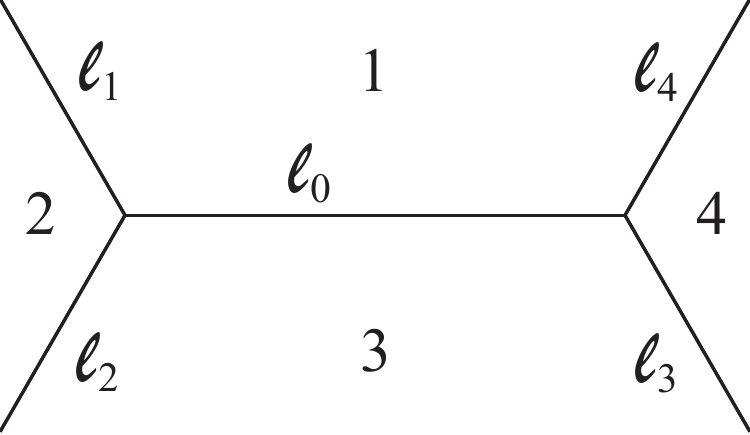, width=1.5in}}
\caption{The vector field $X_0$.}
\label{fig:vfield}
\end{figure}

Define the vector field
\begin{equation}
\label{eq:X}
X_0=-\frac{\partial}{\partial \ell_1}
+\frac{\partial}{\partial \ell_2}
-\frac{\partial}{\partial \ell_3}
+\frac{\partial}{\partial \ell_4} .
\end{equation}
 We then have
 \begin{align*}
 \iota_{X_0}\big(p_1^2\omega_1\big)\big|_{RG_{g,n}(\bp)}
  &= -d\ell_0
  -d\ell_1   -d\ell_4
 \\
 \iota_{X_0}\big(p_2^2\omega_2\big)\big|_{RG_{g,n}(\bp)}
  &=d\ell_1+d\ell_2
 \\
 \iota_{X_0}\big(p_3^2\omega_3\big)\big|_{RG_{g,n}(\bp)}
  &= -d\ell_0
  -d\ell_2   -d\ell_3
\\
\iota_{X_0}\big(p_4^2\omega_4\big)\big|_{RG_{g,n}(\bp)}
 &=d\ell_3+d\ell_4.
 \end{align*}
 Therefore, 
 $$
 \iota_{X_0}\omega_K(\bp)\big|_{RG_{g,n}(\bp)}
  = -2d\ell_0
  $$
 on the tangent space $T_\Gamma RG_{g,n}(\bp)$. This shows that 
 the $2$-form $\omega_K(\bp)$ restricted on $\Ker(d\pi_\Gamma)$ is a linear isomorphism. We refer to \cite{BCSW} for more
 detail.

 Alternatively, we can introduce the symplectic structure
 on $RG_{r,n}(\bp)$ through symplectic reduction.
 The ribbon graph complex $RG_{g,n}$ comes with a natural
 fibration on it, the \emph{tautological torus bundle}
 \begin{equation}
 \begin{CD}
 \bT @>{\mu}>> \bR_+ ^n\\
 @V{\tau}VV\\
 RG_{g,n}
 \end{CD}.
 \end{equation}
 The fiber of $\tau$ at a metric ribbon graph $\Gamma$ is
 the cartesian product 
 of the boundary of the $n$ faces of $\Gamma$, which is identified
 with the collection of $n$ polygons. Topologically each fiber of
 $\bT$
 is an $n$-dimensional torus $T^n = (S^1)^n$. We
 use the same letter $\bT$ for the total space of  this torus bundle,
 whose dimension $2(3g-3+2n)$ is always even.
 
 The identification of the $i$-th face of $\Gamma\in
 RG_{r,n}(\bp)$
  and the circle $S^1 
 = \bR/\bZ$ is given as follows. First we choose a vertex on
 the $i$-th polygon, and name the edges on the $i$-th face
 as $1,2,\dots,k$ in this cyclic order such that the chosen 
 vertex is the beginning point of edge $1$ and the end point of 
 edge $k$.
 Let $\ell_\a$ be the length of edge $\a$ as before. We choose
 a parameter $\phi_i$ subject to $0\le \phi_i\le \ell_1$. 
  Under the re-naming of the edges $(1,2,\dots,k)\longmapsto
 (2,3,\dots,k,1)$, $\phi_i$ changes to $\phi_i '=\phi_i+\ell_1$.
 The choice of the vertex and $\phi_i$ is identified with an element 
 of $S^1$, and also determines the torus action on the 
 fibration $\bT$.

 Define a $2$-form $\Omega$ by
 \begin{equation}
 \label{eq:Omega}
 \begin{aligned}
 \Omega &=\sum_{i=1} ^n \hat\omega_i\\
 \hat\omega_i&=
 \sum_{\a<\b} d\ell_\a \wedge d\ell_\b + d\left(\frac{\phi_i}{p_i}
 \right)
 \wedge d(p_i ^2).
 \end{aligned}
 \end{equation}
The cyclic re-naming of edges changes $\hat\omega_i$ to 
 $$
\hat\omega_i ' = \sum_{2\le \a<\b} d\ell_\a\wedge d\ell_\b
 + \sum_{2\le \a} d\ell_\a \wedge d\ell_1 +
 d\left(\frac{\phi_i+\ell_1}{p_i}
 \right)
 \wedge d(p_i ^2).
 $$
 Therefore, 
 $$
\hat \omega_i-\hat\omega_i ' = 
 \sum_{2\le \b} d\ell_1 \wedge d\ell_\b 
 - \sum_{2\le \a} d\ell_\a \wedge d\ell_1 +
2d\ell_1 \wedge dp_i=0,
$$
and hence $\Omega$ is a globally well-defined $2$-form
on the total space $\bT$.
The moment map of the torus action on $\bT$
is the assignment
$$
\mu: \bT\owns (\Gamma,\phi_1,\dots,\phi_n)
\longmapsto (p_1 ^2,\dots,p_n ^2)\in \bR_+ ^n.
$$
The symplectic quotient 
$\mu^{-1}(L)/\!/T^n$
of $\bT$ by this torus action 
 is 
$\big(RG_{g,n}(\bp),\omega_K(\bp)\big)$ of (\ref{eq:omegap}).

Now we define the symplectic volume of the moduli space
$\cM_{g,n}\isom RG_{g,n}(\bp)$ by
\begin{equation}
\label{eq:Svolume}
v^S _{g,n}(\bp)= \int_{RG_{g,n}(\bp)} \exp\big(\omega_K(\bp)
\big).
\end{equation}
 Applying the recursion  argument similar to our proof
 of Theorem~\ref{thm:integralrecursion}
 to the symplectic reduction 
 of $RG_{g,n}$ by the torus action,
the following theorem was established in \cite{BCSW}.
 
 \begin{thm}[\cite{BCSW}]
 \label{thm:Srecursion}
 The symplectic volume satisfies the following topological 
 recursion.
  \begin{multline}
  \label{eq:Srecursion}
p_1 v_{g,n} ^S(p_N) 
= 
\sum_{j=2}^n 
\Bigg[
\int_{0}^{p_1 + p_j} q(p_1 + p_j - q)
 v_{g,n-1}^S(q,p_{N\setminus \{1,j\}}) dq
 \\
 +
 H(p_1-p_j)
\int_{0}^{p_1 - p_j} q(p_1 - p_j - q)
 v_{g,n-1}^S(q,p_{N\setminus \{1,j\}}) dq
\\
-
H(p_j-p_1)
\int_{0}^{p_j - p_1} q(p_j - p_1 - q)
 v_{g,n-1}^S(q,p_{N\setminus \{1,j\}}) dq
 \Bigg]
 \\
 +
2  \iint _{0 \leq q_1+q_2 \leq p_1} q_1q_2(p_1 - q_1-q_2)
 \Bigg[
 v_{g-1,n+1}^S(q_1,q_2,p_{N\setminus\{1\}})
  \\
 + \sum_{\substack{g_1 + g_2 = g \\ 
 {I} \sqcup {J} = N\setminus\{1\}}} ^{\rm{stable}}
v_{g_1,|I|+1}^S(q_1,p_{{I}})v_{g_2,|J|+1}(q_2,p_{{J}})dq_1dq_2
\Bigg].
\end{multline}
 \end{thm}

The initial values are easy to calculate. For the case
of $(g,n) = (0,3)$, since the perimeter $(p_1,p_2,p_3)\in\bR_+^3$
determines the length of each edge, the symplectic form is
$1$ on a single point. Thus we have 
\begin{equation}
\label{eq:S03}
v_{0,3}^S(p_1,p_2,p_3) = 1.
\end{equation}
The unique trivalent graph of type $(g,n) = (1,1)$ is
given in Figure~\ref{fig:s11}, which has the automorphism
group $\bZ/6\bZ$. The perimeter map is
given by $p=2(\ell_1+\ell_2+\ell_3)$. The restriction of 
$\omega_K(p)$ on $RG_{1,1}(p)$ is $2d\ell_1\wedge d\ell_2$.
Therefore, we have
\begin{equation}
\label{eq:S11}
v_{1,1}^S(p)
= \frac{1}{6}\int_{0\le \ell_1+\ell_2\le\frac{p}{2}}
2d\ell_1 \wedge d\ell_2
=\frac{1}{24}\;p^2.
\end{equation}

\begin{figure}[htb]
\centerline{\epsfig{file=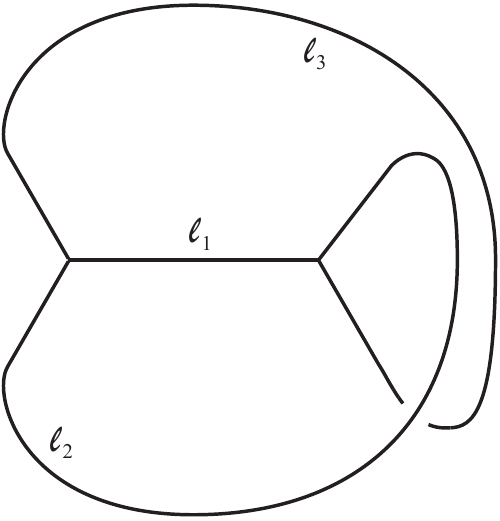, height=1.2in}
\hskip1in\epsfig{file=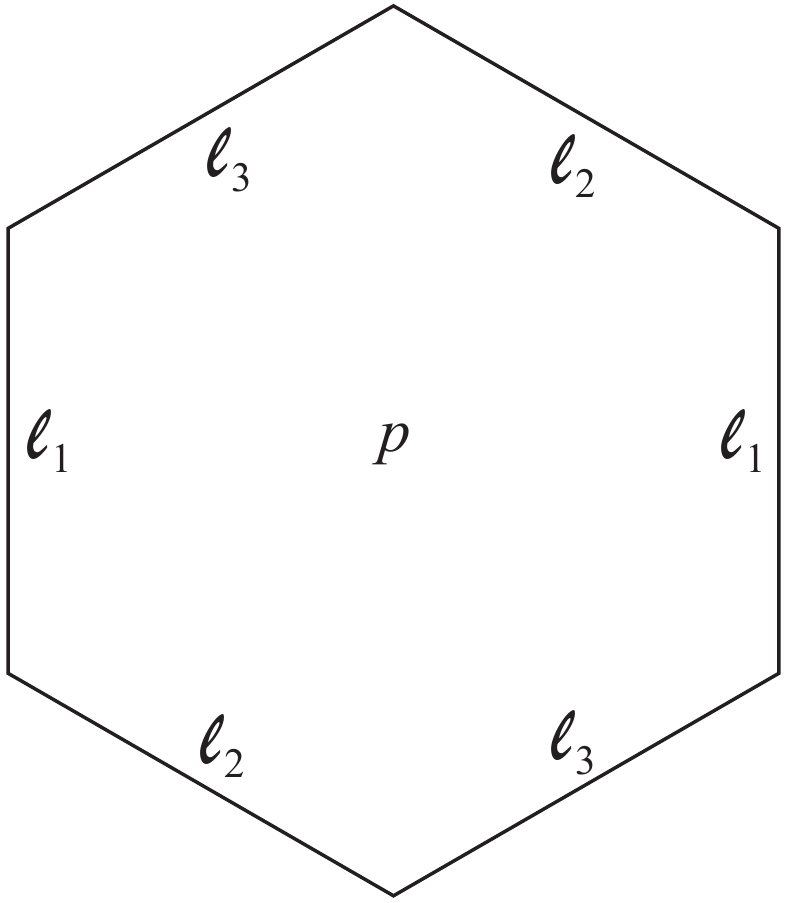, height=1.2in}}
\caption{The trivalent ribbon graph of type $(1,1)$.}
\label{fig:s11}
\end{figure}

 We now consider the Laplace transform of the symplectic 
 volume
 $v_{g,n}^S(\bp)$.
 
 \begin{thm}
 \label{thm:LTK}
The symmetric function $V_{g,n}^S(t_N)$
defined by the Laplace transform
  \begin{equation}
  \label{eq:LTK}
 	V_{g,n}^S(t_1, \dots, t_n)dt_1\tensor \cdots
	\tensor dt_n 
	\overset{\rm{def}}{=} 
	d_1\tensor\cdots\tensor d_n 
	\int_{\bR_+ ^n}
	  v_{g,n}^S(\bp)
	  e^{-\la w,\bp \ra}dp_1\cdots dp_n
	 \end{equation}
and the coordinate change
\begin{equation}
\label{eq:w=2/t}
w_j=-\frac{2}{t_j}
\end{equation}	 
 satisfies the topological recursion 
 \begin{multline}
 \label{eq:LTSrecursion}
V_{g,n}^S(t_N) = 
-\frac{1}{4}\sum_{j=2} ^\infty
\frac{\partial}{\partial t_j}
\left[
\frac{t_j}{t_1^2-t_j^2}
\bigg(
t_1^4V_{g,n-1}^S(w_{N\setminus\{j\}})
-
t_j^4V_{g,n-1}^S(w_{N\setminus\{1\}})
\bigg)
\right]
\\
-
\frac{1}{4}\;t_1^4
 \left(
  V_{g-1,n+1}^S(t_1,t_1, t_{N\setminus\{1\}}) 
  + \sum_{\substack{g_1 + g_2 = g, \\
  {I} \sqcup {J} = N\setminus\{1\}}}  
  V_{g_1,|I|+1}^S(t_1, t_{{I}}) 
  V_{g_2,|J|+1}^S(t_1, t_{{J}}) 
  \right).
\end{multline}
  \end{thm}
  
  \noindent
The proof of this theorem is given in 
Appendix~\ref{app:LTProof}. 
The very reason that Kontsevich was interested in 
the symplectic volume of the moduli space is that it
gives the generating function of the 
intersection numbers (\ref{eq:tau}) 
\begin{equation}
\label{eq:intersection}
V_{g,n}^S(t_N) =
(-1)^n \sum_{\substack{d_1+\cdots d_n\\=3g-3+n}}
\la \tau_{d_1}\cdots\tau_{d_n}\ra_{g,n}
\prod_{j=1} ^n (2d_j+1)!!\left(\frac{t_j}{2}\right)^{2d_j}.
\end{equation}
The topological recursion (\ref{eq:LTSrecursion})
produces a relation among the coefficients, which is 
known as the DVV formula of \cite{DVV}, and is 
equivalent to the Virasoro constraint condition of
\cite{W1991}.

Since the volume is for $\cM_{g,n}$ and the 
intersection numbers are for $\Mbar_{g,n}$,
it is not obvious why they are the same thing.
From the  deep theory of Mirzakhani \cite{Mir1, Mir2},
it becomes obvious why and how they are 
related.

We are now ready to calculate the Kontsevich constants.

\begin{thm}
\label{thm:KinLT}
The ratio of the two volume polynomials 
$V_{g,n}^S(t_N)$ and $V_{g,n}^E(t_N)$ 
is a
 constant  depending only on $g$ and $n$:
\begin{equation}
\label{eq:rho}
\rho_{g,n}(t) \overset{\rm{def}}{=}
\frac{V_{g,n}^S(t_N)}{V_{g,n}^E(t_N)}
= 2^{5g-5+2n}.
\end{equation}
\end{thm}

\begin{proof}
We use induction on $2g-2+n$. From (\ref{eq:S03}), 
(\ref{eq:S11}), (\ref{eq:L03}) and (\ref{eq:cL11}), we have
\begin{equation}
\label{eq:initial}
\begin{cases}
V_{0,3}^E(t_1,t_2,t_3)=-\frac{1}{16}\\
V_{0,3}^S(t_1,t_2,t_3)=-\frac{1}{8}
\end{cases}
\qquad{\text{and}}\qquad
\begin{cases}
V_{1,1}^E(t)=-\frac{1}{128}\; t^2\\
V_{1,1}^S(t)=-\frac{1}{32}\;t^2
\end{cases}.
\end{equation}
Thus the initial values satisfy (\ref{eq:rho}). 
We observe that the recursion formulas
(\ref{eq:LTErecursion}) and
 (\ref{eq:LTSrecursion}) are the same except for the 
 constant factors on the first and the second lines of the right-hand
 side.
Therefore, 
if we changed $V_{g,n}^E(t_N)$ to $2^{5g-5+2n}\cdot
V_{g,n}^E(t_N)$ in (\ref{eq:LTErecursion}), then its
recursion formula would
become identical to (\ref{eq:LTSrecursion}).
Since the recursion uniquely determines all values for 
$(g,n)$ subject to $2g-2+n>0$ from the initial
values (\ref{eq:initial}), we establish
(\ref{eq:rho}). This completes the proof.
\end{proof}

\section{The Eynard-Orantin theory on $\bP^1$}
\label{sect:EO}

The number of integral ribbon graphs $N_{g,n}(\bp)$ is
a difficult function to deal with because it is not given by
a single formula. As we have noted, it behaves like a polynomial
for large $\bp\in\bZ_+^n$, while it takes value $0$ whenever
$p_1+\cdots+p_n$ is odd. Compared to this, the 
Laplace transformed function such as $\cL_{g,n}(t_N)$
is a far nicer object. Indeed $\cL_{g,n}(t_N)$ is a 
Laurent polynomial and satisfies a simple differential 
recursion formula (\ref{eq:LTrecursion}). We also note that
 the recursion formulas 
(\ref{eq:LTrecursion}), 
(\ref{eq:LTErecursion}),
and (\ref{eq:LTSrecursion})
take a very similar shape. 
Over the years several authors (including
\cite{BKMP, BM, DV2007, E2007, EMS, EO1, EO2, 
K1992, M,
MZ,OP1,Z1,Z2})
have noticed 
 that many different combinatorial structures
(on the A-model side of a topological string theory)
can be uniformly treated on the B-model side, after
taking the Laplace transform. The importance of the 
Laplace transform as the mirror map was 
noted in \cite{EMS}.
This uniform structure after the Laplace 
transform is the 
manifestation of the Eynard-Orantin theory.
We will show in this section that the recursions
(\ref{eq:LTrecursion}), 
(\ref{eq:LTErecursion}),
and (\ref{eq:LTSrecursion})
become identical under the formalism proposed in 
\cite{EO1}.

We are not in the place to formally present the
Eynard-Orantin formalism in an axiomatic way.
Instead of giving the full account, we are satisfied with
explaining
a limited case when the \emph{spectral curve}
of the theory is $\bP^1$. 
The word ``spectral curve'' was used in \cite{EO1}
 because of  the analogy of the
spectral curves appearing in the Lax formalism of
 integrable systems.

We start with the spectral curve $C=\bP^1\setminus S$,
where $S\subset \bP^1$ is a finite set.
We also need two generic elements $x$ and $y$
of $H^0(C,\cO_C)$, where $\cO_C$ denotes the sheaf of 
holomorphic functions on $C$. The condition we impose 
on $x$ and $y$ is
that the holomorphic maps
\begin{equation}
\label{eq:x,y}
x:C\longrightarrow \bC \qquad {\text{and}}
\qquad 
y:C\longrightarrow \bC 
\end{equation}
have only simple ramification points, i.e., their derivatives 
$dx$ and $dy$ have simple zeros, and that
\begin{equation}
\label{eq:(x,y)}
(x,y):C\owns t\longmapsto \big(x(t),y(t)\big)\in \bC^2
\end{equation}
is an immersion. Let $\Lambda^1(C)$ denote the
sheaf of meromorphic $1$-forms on $C$, and
\begin{equation}
\label{eq:Hn}
H^n = H^0\big(C^n,\Sym^n(\Lambda^1(C))\big)
\end{equation}
the space of meromorphic symmetric differentials of
degree $n$. The Cauchy differentiation kernel is an example
of such differentials:
\begin{equation}
\label{eq:W02}
W_{0,2}(t_1,t_2) = \frac{dt_1\tensor dt_2}{(t_1-t_2)^2}
\in H^2.
\end{equation}
In the literatures starting from \cite{EO1},
 the Cauchy differentiation kernel 
 has been called the
 \emph{Bergman kernel},
 even thought it has nothing to do with 
 the Bergman kernel in complex analysis. 
A bilinear operator
\begin{equation}
\label{eq:K}
K:H\tensor H\longrightarrow H
\end{equation}
naturally extends to 
\begin{multline*}
K:H^{n_1+1}\tensor H^{n_2+1}
\owns (f_0,f_1,\dots,f_{n_1})
\tensor
 (h_0,h_1,\dots,h_{n_2})
 \\
\longmapsto
(K(f_0,h_0), f_1,\dots,f_{n_1},h_1,\dots,h_{n_2})
\in
H^{n_1+n_2+1}
\end{multline*}
$$
K:H^{n+1}
\owns (f_0,f_1,\dots,f_{n_1})
\longmapsto 
(K(f_0,f_1), f_2,\dots,f_{n_1})
\in 
H^n.
$$
Suppose we are given an infinite sequence 
 $\{W_{g,n}\}$ of differentials $W_{g,n}\in H^n$
 for all $(g,n)$ subject to the stability condition $2g-2+n>0$.
 We say \emph{this sequence satisfies a topological  recursion 
 with respect to the kernel} $K$ if
 \begin{equation}
 \label{eq:EOgeneral}
 W_{g,n}= K(W_{g,n-1},W_{0,2})
 + K(W_{g-1,n+1}) 
 +\half
 \sum_{\substack{g_1+g_2=g\\I\sqcup J=N\setminus\{1\}}}
 ^{\text{stable}}
 K\!\left(W_{g_1,|I|+1},W_{g_2,|J|+1}\right).
 \end{equation}
The characteristic of the Eynard-Orantin theory lies in the
particular choice of the \emph{Eynard kernel} that 
reflects the parametrization (\ref{eq:(x,y)}) and the 
ramified coverings (\ref{eq:x,y}).
Let $\cA=\{a_1,\dots,a_r\}\subset C$ be 
the set of simple ramification 
points of the $x$-projection map.
Since locally at each $a_\lam$ the 
$x$-projection is a double-sheeted covering,  we can choose the
deck transformation map 
\begin{equation}
\label{eq:slam}
s_\lam :U_\lam\overset{\sim}{\longrightarrow}U_\lam,
  \end{equation}
where $U_\lam\subset C$ is an appropriately
chosen simply connected neighborhood of $a_\lam$.

\begin{Def}
\label{def:EynardKernel}
The \emph{Eynard kernel} is the linear map
$H\tensor H\rightarrow H$ defined by
\begin{multline}
\label{eq:EK}
K\big(f_1(t_1)dt_1,f_2(t_2)dt_2\big)
\\
=
\frac{1}{2\pi i}
\sum_{\lam=1} ^r
\oint_{|t-a_\lam|<\epsilon}
K_\lam (t,t_1)\bigg(f_1(t)dt\tensor f_2\big(s_\lam(t)\big)ds_\lam(t)
+
f_2(t)dt\tensor f_1\big(s_\lam(t)\big)ds_\lam(t)
\bigg),
\end{multline}
where
\begin{equation}
\label{eq:Ekernel}
K_\lam(t,t_1) =
\half 
\left(
\int_{t} ^{s_\lam(t)} W_{0,2}(t,t_1)\;dt
\right)
\tensor dt_1\cdot 
 \frac{1}{\bigg(y(t)-y\big(s_\lam(t)\big)\bigg)dx(t)},
\end{equation}
and $\frac{1}{dx(t)}$ is 
the contraction operator with respect to the vector field
$$
\left(\frac{dx}{dt}\right)^{-1}
\frac{\partial}{\partial t}. 
$$
The integration is taken with respect to the
$t$-variable along a small loop around $a_\lam$
that contains no singularities other than $t=a_\lam$. 
A topological recursion with respect to the Eynard kernel 
is what we call the \emph{Eynard-Orantin recursion}
in this paper.
\end{Def}

To convert (\ref{eq:LTrecursion}) to 
the Eynard-Orantin formalism, we need to identify
the spectral curve of the theory and the unstable
case $\cL_{0,2}(t_1,t_2)$. 
The spectral curve is   a plane algebraic curve
\begin{equation}
\label{eq:C}
C = \left\{(x,y)\in \bC^2\;\left|\; xy=
y^2+1\right.\right\},
\end{equation}
which is the same curve considered in \cite{N2}. Here we 
introduce a different parametrization
\begin{equation}
\label{eq:spectral}
\begin{aligned}
x(t)&=\frac{t+1}{t-1}+\frac{t-1}{t+1}= 2+\frac{4}{t^2-1}\\
y(t)&=\frac{t+1}{t-1}
\end{aligned}
\end{equation}
with a parameter $t\in\bP^1\setminus\{1,-1\}$
so that the resulting differentials become Laurent polynomials.
This use of the parametrization is similar to that of
\cite{EMS, MZ}.
The $x$-projection 
\begin{equation}
\label{eq:xprojection}
\pi:C\owns t \longmapsto x(t)
\in  \bC
 \end{equation}
has simple ramification points at $t=0$
and $t=\infty$, since
$$
dx = -\frac{8t}{(t^2-1)^2}\; dt.
$$
We note that since the  map $\pi$
is globally a branched double-sheeted
 covering, its covering transformation 
is globally defined and is given by 
\begin{equation}
\label{eq:s}
s: C\owns t\longmapsto s(t)=-t\in C.
\end{equation}

\begin{figure}[htb]
\centerline{\epsfig{file=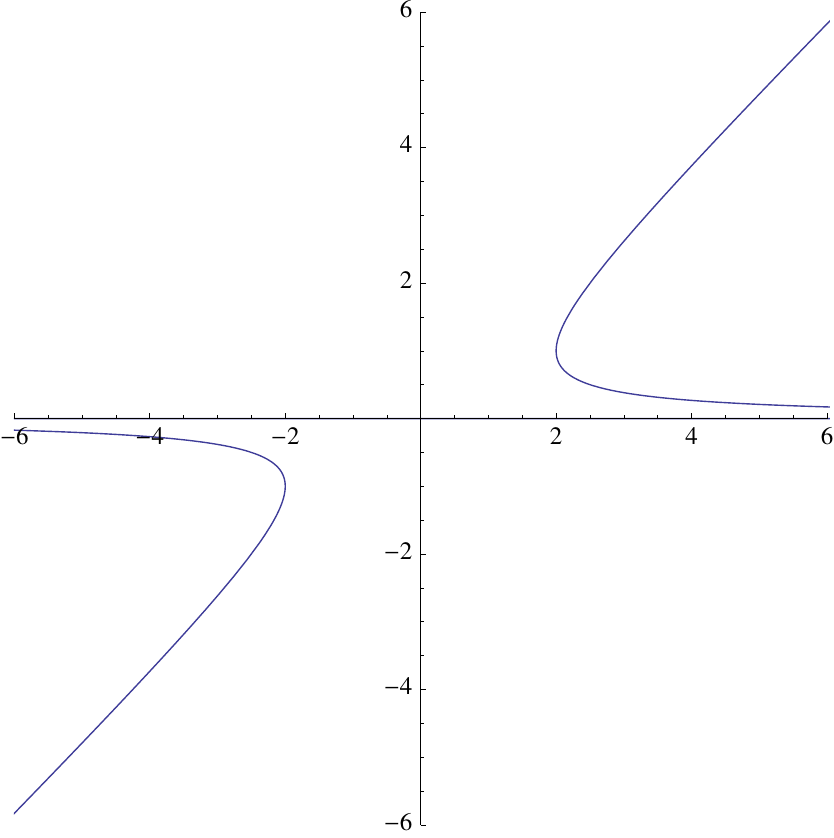, width=1.5in}}
\caption{The spectral curve $x=y+\frac{1}{y}$.}
\label{fig:spectral}
\end{figure}

The unstable $(0,2)$ case
is calculated in 
Appendix~\ref{app:examples}, (\ref{eq:L02}). The result is
$$
\cL_{0,2}(t_1,t_2)dt_1\tensor dt_2
=\frac{dt_1\tensor dt_2}{(t_1+t_2)^2}=
\frac{dt_1\tensor dt_2}{\big(t_1-s(t_2)\big)^2}.
$$
This quadratic differential form plays the role of the
Cauchy differentiation kernel.
For every holomorphic differential $f(t)dt$ on $C$, we have
\begin{multline}
\label{eq:diff-kernel}
-\frac{1}{2\pi i}\oint \frac{1}{dt}
\bigg[
 f\big(s(t)\big)ds(t)
\cL_{0,2}(t,t_1)dt\tensor dt_1
+
f(t)dt\cL_{0,2}\big(s(t),t_1\big) ds(t)\tensor dt_1
\bigg]
\\
=
\left(
\frac{1}{2\pi i}\oint 
\frac{f(-t)}{(t+t_1)^2}dt 
\right)
\tensor dt_1
+
\left(
\frac{1}{2\pi i}\oint 
\frac{f(t)}{(t-t_1)^2}dt 
\right)
\tensor dt_1
=2f'(t_1)dt_1,
\end{multline}
where the operation $\frac{1}{dt}$ is the contraction 
by the vector field $\frac{\partial}{\partial t}$, and the
integration is taken with respect to $t$ along a 
positively oriented simple loop that 
contains both $t_1$ and $s(t_1)$.  Actually, the 
contour integral should be considered as
the residue calculation at $t=\infty$ with respect
to the opposite orientation. This explains the 
minus sign in (\ref{eq:diff-kernel}).

\begin{thm}
\label{thm:EO}
The topological recursion {\rm{(\ref{eq:LTrecursion})}}
is equivalent to the Eynard-Orantin recursion of {\rm{\cite{EO1}:}}
\begin{multline}
\label{eq:EO}
\cL_{g,n}(t_N)dt_N
\\
=
\frac{1}{2\pi i}\int _\Gamma K(t,t_1)
\Bigg[
\sum_{j=2} ^n
\bigg(
\cL_{g,n-1}(t,t_{N\setminus\{1,j\}})dt\tensor 
dt_{N\setminus\{1,j\}}\tensor 
\cL_{0,2}\big(s(t),t_j\big)ds(t)\tensor dt_j
\\
+
\cL_{g,n-1}\big(s(t),t_{N\setminus\{1,j\}})ds(t)\tensor 
dt_{N\setminus\{1,j\}}\tensor 
\cL_{0,2}(t,t_j)dt\tensor dt_j
\bigg)
\\
+
\cL_{g-1,n+1}\big(t,s(t),t_{N\setminus\{1\}}\big)dt
\tensor ds(t)\tensor 
dt_{N\setminus\{1\}}
\\
+
\sum_{\substack{g_1+g_2=g\\I\sqcup J=N\setminus\{1\}}}
^{\rm{stable}}
\bigg(\cL_{g_1,|I|+1}(t,t_I)dt\tensor dt_I\bigg)
\tensor
\bigg(\cL_{g_2,|J|+1}\big(s(t),t_J\big)ds(t)\tensor dt_J\bigg)
\Bigg].
\end{multline}
Here the contour integration 
 is taken with respect to $t$ along a curve 
$\Gamma$ that consists of a large
circle of the negative orientation
centered at the origin with radius $r>\max_{j\in N} |t_j|$,
and  a small circle around the origin
of the positive orientation. 
We use a simplified 
notation  $dt_I = \bigotimes_{i\in I}dt_i$ for $I\subset N$.
\end{thm}

\begin{figure}[htb]
\centerline{\epsfig{file=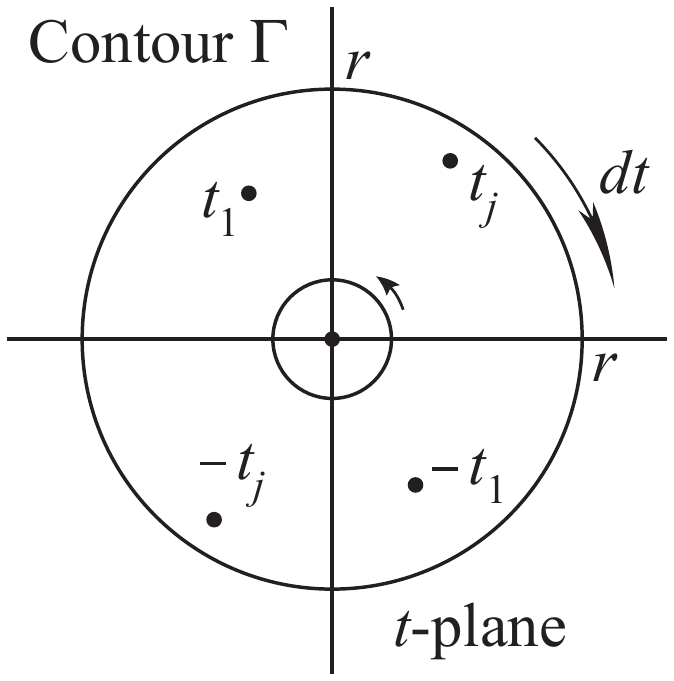, width=1.5in}}
\caption{The integration contour $\Gamma$. This contour 
encloses an annulus bounded by two concentric 
circles centered at the origin. The outer one has a large radius 
$r>\max_{j\in N} |t_j|$ and the negative orientation,
 and the inner one has an infinitesimally small radius with 
 the positive
 orientation.}
\label{fig:contour}
\end{figure}

\begin{rem}
\begin{enumerate}

\item The contour integral (\ref{eq:EO}) 
can be phrased as the sum of the \emph{residues}
of the integrand at the raminfication points of 
the spectral curve $t=0$ and $t=\infty$, which is 
the language used in 
\cite{EO1}.
\item
The first and the second lines of 
the right-hand side of (\ref{eq:EO}) are
unstable $(0,2)$ cases of the fourth line when we have 
$(g_1,I) = (0,\{j\})$ or $(g_2,J) = (0,\{j\})$.

\item
In terms of the Cauchy differentiation kernel $W_{0,2}(t,t_1)$
of (\ref{eq:W02}),
we have
$$
\cL_{0,2}(t,t_1)dt\tensor dt_1 = W_{0,2}(t,t_1)
-\pi^*W_{0,2}(x,x_1)
$$
as proved in Appendix~\ref{app:examples},
(\ref{eq:02=B}).
Since $\pi^*W_{0,2}(x,x_1)$ 
is invariant under the deck transformation
$s:C\rightarrow C$ applied to the entry $x$, 
and since $\cL_{g,n}(t_N)$ is an even function by
Theorem~\ref{thm:LTrecursion}, 
we can replace 
$\cL_{0,2}(t,t_1)dt\tensor dt_1$ with  $W_{0,2}(t,t_1)$
in (\ref{eq:EO}).

\end{enumerate}

\end{rem}

\begin{proof}
The Eynard kernel of our setting  is
\begin{multline*}
K(t,t_1) =
\half 
\left(
\int_{t} ^{s(t)} \cL_{0,2}(t,t_1)\;dt
\right)
\tensor dt_1\cdot 
 \frac{1}{\bigg(y(t)-y\big(s(t)\big)\bigg)dx(t)}
 \\
=
\half 
\left(
\int_{t} ^{s(t)} \frac{dt}{(t+t_1)^2}
\right)\tensor dt_1\cdot
\frac{1}{\left(\frac{t+1}{t-1}-\frac{-t+1}{-t-1}\right)
\frac{\partial}{\partial t}\left(
\frac{t+1}{t-1}+\frac{t-1}{t+1}
\right)}\cdot\frac{1}{dt}
\\
=
-\half
\left(
\frac{1}{t-t_1}+\frac{1}{t+t_1}
\right)\frac{1}{32}\;\frac{(t^2-1)^3}{t^2}
\cdot \frac{1}{dt}\tensor dt_1.
\end{multline*}
Thus 
for any symmetric Laurent polynomial 
$f(t,s)$ in  $t^2$ and $s^2$, we have
$$
\frac{1}{2\pi i}\int_{\Gamma}
 K(t,t_1)f\big(t,s(t)\big)dt\tensor ds(t)
=
-
f(t_1,t_1)\;\frac{1}{32}\;\frac{(t_1^2-1)^3}{t_1^2}\;
dt_1,
$$
since $s(t) = -t$. Therefore, 
the third and the fourth lines of the right-hand side of
(\ref{eq:EO}) becomes
\begin{multline*}
-\frac{1}{32}\; \frac{(t_1^2-1)^3}{t_1^2}
\Bigg[
\cL_{g-1,n+1}(t_1,t_1,t_{N\setminus\{1\}})
\\
+\sum_{\substack{g_1+g_2=g\\I\sqcup J=N\setminus\{1\}}}
^{\rm{stable}}
\cL_{g_1,|I|+1}(t_1,t_I)\cL_{g_2,|J|+1}(t_1,t_J)
\Bigg]
\tensor dt_N.
\end{multline*}
This is 
because $\cL_{g,n}(t_N)$ for  $(g,n)$ in the stable range is
a Laurent polynomial in $t_1^2,\dots,t_n^2$, hence the only 
simple poles
in the complex $t$-plane within the contour $\Gamma$
 of (\ref{eq:EO}) that appear in the third and fourth
lines are located at $t=t_1$ and $t=-t_1$.

Even though the first and the second lines of 
the right-hand side of (\ref{eq:EO}) are
somewhat a degenerate case of the fourth line as remarked above,
the analysis becomes different because 
$\cL_{0,2}(t,t_j)$
contributes new poles in the $t$-plane. First we note that
\begin{multline}
\label{eq:gn-1}
\cL_{g,n-1}(t,t_{N\setminus\{1,j\}})dt\tensor 
dt_{N\setminus\{1,j\}}\tensor 
\cL_{0,2}\big(s(t),t_j\big)ds(t)\tensor dt_j
\\
+
\cL_{g,n-1}\big(s(t),t_{N\setminus\{1,j\}})ds(t)\tensor 
dt_{N\setminus\{1,j\}}\tensor 
\cL_{0,2}(t,t_j)dt\tensor dt_j
\\
=
-\cL_{g,n-1}(t,t_{N\setminus\{1,j\}})
\left(
\frac{1}{(-t+t_j)^2}+\frac{1}{(t+t_j)^2}
\right)
dt^{\tensor 2} \tensor dt_{N\setminus\{1\}}
\\
=
-2\cL_{g,n-1}(t,t_{N\setminus\{1,j\}})
\frac{\partial}{\partial t_j}\;
\frac{t_j}{t^2-t_j^2}
\;
dt^{\tensor 2} \tensor dt_{N\setminus\{1\}}.
\end{multline}
Apply the operation $\frac{1}{2\pi i}\int_{\Gamma} K(t,t_1)$
to 
(\ref{eq:gn-1})
and collect the residues at $t= t_1$ and $t=-t_1$. 
We then  obtain
$$
-\frac{1}{16}
\;\frac{\partial}{\partial t_j}
\left[
\frac{t_j}{t_1^2-t_j^2}\;
\frac{(t_1^2-1)^3}{t_1^2}\;
\cL_{g,n-1}(t_{N\setminus\{j\}})
\right]
dt_N.
$$
When $t\sim t_j$ or $t\sim -t_j$, we use
(\ref{eq:diff-kernel}) to derive
\begin{multline*}
-\frac{1}{2\pi i}\int_{\Gamma} \frac{t}{t^2-t_1^2}
\;
\frac{1}{32}\;
\frac{(t^2-1)^3}{t^2}\cdot \frac{1}{dt}\tensor dt_1
\\
\tensor (-1)
\cL_{g,n-1}(t,t_{N\setminus\{1,j\}})
\left(
\frac{1}{(-t+t_j)^2}+\frac{1}{(t+t_j)^2}
\right)
dt^{\tensor 2} \tensor dt_{N\setminus\{1\}}
\\
=-\frac{1}{32}\;
\frac{\partial}{\partial t_j}
\left[
\frac{t_j}{t_j^2-t_1^2}\;\frac{(t_j^2-1)^3}{t_j^2}
\cL_{g,n-1}(t_j,t_{N\setminus\{1,j\}})
\right] dt_N
\\
-
\frac{1}{32}\;
\left(
-\frac{\partial}{\partial t_j}
\right)
\left[
\frac{-t_j}{t_j^2-t_1^2}\;\frac{(t_j^2-1)^3}{t_j^2}
\cL_{g,n-1}(-t_j,t_{N\setminus\{1,j\}})
\right] dt_N
\\
=\frac{1}{16}\;
\frac{\partial}{\partial t_j}
\left[
\frac{t_j}{t_1^2-t_j^2}\;\frac{(t_j^2-1)^3}{t_j^2}
\cL_{g,n-1}(t_j,t_{N\setminus\{1,j\}})
\right] dt_N.
\end{multline*}
This completes the proof.
\end{proof}

Here we note that the spectral curve (\ref{eq:spectral}),
and hence the topological recursion theory of our case,
has a non-trivial automorphism. 
It is given by the transformation
\begin{equation}
\label{eq:tto1/t}
t\longmapsto \frac{1}{t},
\end{equation}
which induces an automorphism
\begin{equation}
\label{eq:Cauto}
C\owns (x,y)\longmapsto (-x,-y)\in C
\end{equation}
of the spectral curve.
It interchanges the two
ramification points of Figure~\ref{fig:spectral}. 
Let 
$$
u=\frac{1}{t},\qquad u_j=\frac{1}{t_j} \qquad \text{for }
j=1,2,\dots,n.
$$
Then we have 
$$
\cL_{0,2}(t,t_j)dt\tensor dt_j=\cL_{0,2}(u,u_j)du\tensor du_j,
$$ 
and $ydx = (-y)d(-x)$. It follows that $K(t,t_1) = K(u, u_1)$,
and we have $\bZ/2\bZ$ as the 
automorphism group of the theory. 
Reflecting this automorphism,
the function $\cL_{g,n}(t_N)$ exhibits the following transformation 
property:
\begin{equation}
\label{eq:1/t}
\cL_{g,n}\left(\frac{1}{t_1},\dots,\frac{1}{t_n}\right)
= (-1)^n \cL_{g,n}(t_1,\dots,t_n) t_1^2\cdots t_n^2.
\end{equation}

  The reason that we choose 
$t$ as our preferred parameter rather than 
$1/t$ in (\ref{eq:t}) 
 is to extract the polynomial behavior of the 
 Laplace transform of the Euclidean 
 volume.
As $t\rightarrow \infty$ the spectral curve
$C$ degenerates to a parabola, and 
the theory changes  from  counting the integral
ribbon graphs to  calculating the Euclidean volume, as we
shall see below. 
By the symmetry argument, the $t\rightarrow 0$
limit also deforms $C$ to a parabola. We can see from (\ref{eq:t})
that 
$$
t\rightarrow \infty \qquad
\Longleftrightarrow \qquad
 e^{-w}=\frac{t+1}{t-1}\rightarrow 1\qquad
\Longleftrightarrow \qquad
w\rightarrow 0,
$$
and the $w\rightarrow 0$ behavior of the Laplace transform
represents the Euclidean volume function, as explained in
Section~\ref{sect:Euc}. Even though there is a symmetry
in the $t$-variables, in terms of $w$, we have
$$
t\rightarrow 0 \qquad
\Longleftrightarrow \qquad
 e^{-w}=\frac{t+1}{t-1}\rightarrow -1,
 $$
and this limit does not correspond to bringing the mesh of the
lattice to $0$.

By restricting (\ref{eq:EO}) to the
top degree terms using
$$
\frac{(t^2-1)^3}{t^2} = t^4-3t^2+3-\frac{1}{t^2},
$$
 we obtain the recursion for the Euclidean volume.

\begin{thm}
\label{thm:EvolumeEO}
Define the Eynard kernel for the Euclidean volume 
by
\begin{equation}
\label{eq:KE}
K_E(t,t_1)=
-\half
\left(
\frac{1}{t-t_1}+\frac{1}{t+t_1}
\right)
\frac{1}{32}\; t^4\cdot 
\frac{1}{dt}\tensor dt_1
\end{equation}
on the spectral curve 
$C_E$ defined by the
parametrization
\begin{equation}
\label{eq:asymspectral}
\begin{cases}
x-2=\frac{4}{t^2}\\
y-1= \frac{2}{t}
\end{cases}.
\end{equation}
Then 
the Laplace transformed Euclidean volume function
$V_{g,n} ^E(t_N)$ satisfies an Eynard-Orantin type
recursion
\begin{multline}
\label{eq:EvolumeEO}
V_{g,n} ^E(t_N)
\\
=-\frac{1}{2\pi i}\oint K_E(t,t_1)
\Bigg[
\sum_{j=2} ^n 
\bigg(
V_{g,n-1}^E(t, t_{N\setminus\{1,j\}})
 dt \tensor dt_{N\setminus\{1,j\}}
 \tensor \frac{ds(t)\tensor dt_j}{\big(s(t)+t_j\big)^2}
 \\
+
V_{g,n-1}^E\big(s(t), t_{N\setminus\{1,j\}}\big)
 ds(t) \tensor dt_{N\setminus\{1,j\}}
 \tensor \frac{dt\tensor dt_j}{(t+t_j)^2}
\bigg)
\\
+
V_{g-1,n+1}^E\big(t, s(t),t_{N\setminus\{1\}})
dt\tensor ds(t)\tensor dt_{N\setminus\{1\}}
\\
+
\sum_{\substack{g_1+g_2=g\\I\sqcup J=N\setminus\{1\}}}
^{\rm{stable}}
\left(
V_{g_1,|I|+1}^E(t,t_I)dt\tensor dt_I
\right)
\tensor
\left(
 V_{g_2,|J|+1}^E\big(s(t),t_J\big)ds(t)\tensor dt_J
 \right)
\Bigg].
\end{multline}
Here the integration contour is a positively oriented 
circle of large radius. 
\end{thm}

The geometry behind the
recursion formula (\ref{eq:EvolumeEO})
is the following. The Euclidean volume 
is obtained by extracting the asymptotic behavior of 
$\cL_{g,n}(t_N)$ as $t\rightarrow \infty$. 
The parametersization 
$$
\begin{cases}
x=2+\frac{4}{t^2-1}\\
y=1+\frac{2}{t-1}
\end{cases}
$$
of the spectral curve (\ref{eq:spectral}) of
Figure~\ref{fig:spectral} near $t\sim \infty$
gives a neighborhood of
one of the 
critical points $(x,y)=(2,1)$. 
Thus we \emph{define} a new spectral curve $C_E$ by the
parametrization
(\ref{eq:asymspectral}),
which is simply a parabola $x-2 = (y-1)^2$.
The deck-transformation of the $x$-projection of the
parabola $C_E$ is still given by $t\mapsto s(t) = -t$.
The recipe of (\ref{eq:Ekernel}) then gives
(\ref{eq:KE}),
provided that the unstable $(0,2)$ geometry still gives the
same kernel
\begin{equation}
\label{eq:VE02}
V_{0,2} ^E(t_1,t_2) = \frac{dt_1\tensor dt_2}{(t_1+t_2)^2}.
\end{equation}
The continuum limit of (\ref{eq:N02}) is 
$$
v_{0,2} ^E(p_1,p_2) = \frac{1}{p_1}\delta(p_1-p_2).
$$
We thus calculate
\begin{equation}
\label{eq:v02calc}
\left(
\int_0 ^\infty\!\!\int_0 ^\infty
p_1p_2v_{0,2} ^E (p_1,p_2) e^{-(p_1w_1+p_2w_2)}dp_1dp_2
\right)
dw_1\tensor dw_2
=\frac{dw_1\tensor dw_2}{(w_1+w_2)^2}.
\end{equation}
The coordinate change (\ref{eq:t}) near $t\sim \infty$ becomes
$$
e^{-w} = \frac{t+1}{t-1}\longmapsto
1-w = 1+\frac{2}{t},
$$
i.e., $w=-\frac{2}{t}$. Under this change, which is 
an automorphism of $\bP^1$, (\ref{eq:v02calc})
remains the same, and we obtain (\ref{eq:VE02}).
The $x$-projection of the spectral curve $C_E$ defined by
the parametrization (\ref{eq:asymspectral}) now has
only one ramification point at $t=\infty$. Thus the 
integration contour $\Gamma$ of (\ref{eq:EO}) has changed
into a single large circle in (\ref{eq:EvolumeEO}).

The Eynard-Orantin recursion for the symplectic volume is 
given by the choice of the spectral curve $C_S$ parametrized 
by
\begin{equation}
\label{eq:CS}
\begin{cases}
x=\frac{1}{t^2}\\
y=\frac{1}{t}
\end{cases}.
\end{equation}
Since the curve is isomorphic to 
$\bC$, we use the same Cauchy differentiation
kernel $W_{02}(t_1,t_2)$ of (\ref{eq:W02}) in place
of $V_{0,2}^S(t_1,t_2)$.
The Eynard kernel (\ref{eq:Ekernel}) for this case is
\begin{equation}
\label{eq:SE}
K(t,t_1)=
-\half \left(-\frac{1}{s(t)+t_1}+\frac{1}{t+t_1}\right)
\frac{t^4}{4}\cdot \frac{1}{dt}\tensor dt_1.
\end{equation}
Then the  recursion takes 
exactly the same form of (\ref{eq:EvolumeEO}).

\begin{appendix}

\section{Calculation of the Laplace transforms}
\label{app:LTProof}
\setcounter{section}{1}

In this Appendix we prove the Laplace transform formulas
used in the main text.
We first derive the topological recursion for 
\begin{equation}
\label{eq:hatL}
\widehat{L}_{g,n}(w_1,\dots,w_n) =
\sum_{\bp\in\bZ_{\ge 0}^n} p_1p_2\cdots p_n N_{g,n}(\bp) 
e^{-\la p, w\ra}.
\end{equation}
Since we multiply  the number of
integral ribbon graphs by 
$p_1\cdots p_n$, we can allow 
all non-negative integers $p_j$ in the summation,
which makes our calculations simpler.

\begin{prop}
\label{prop:LTinw}
The Laplace transform $\wL_{g,n}(w_N)$ satisfies the following
topological recursion. 
\begin{multline}
\label{eq:LTinw}
\wL_{g,n}(w_N)
\\
=
\sum_{j=2} ^n \frac{\partial}{\partial w_j}
\left[
\left(
\frac{e^{w_1}}{e^{w_1}-e^{w_j}}-
\frac{e^{w_1+w_j}}{e^{w_1+w_j}-1}
\right)
\left(
\frac{
\wL_{g,n-1}(w_{N\setminus\{j\}})}
{(e^{w_1}-e^{-w_1})^2}
-
\frac{
\wL_{g,n-1}(w_{N\setminus\{1\}})}
{(e^{w_j}-e^{-w_j})^2}
\right)
\right]
\\
+
\frac{1}{(e^{w_1}-e^{-w_1})^2}
\Bigg[
\wL_{g-1,n+1}(w_1,w_1,w_{N\setminus\{1\}})
\\
+
\sum_{\substack{g_1+g_2=g\\I\sqcup J=N\setminus\{1\}}}
^{\rm{stable}}
\wL_{g_1,|I|+1}(w_1,w_I)\wL_{g_2,|J|+1}(w_1,w_J)
\Bigg].
\end{multline}
\end{prop}

\begin{proof}
First we multiply both sides of   (\ref{eq:integralrecursion})
by $p_2p_3\cdots p_n$
and compute its Laplace transform. The left-hand side gives
$\wL_{g,n}(w_N)$.

The first line of the right-hand side is
\begin{multline*}
\sum_{j=2} ^n \sum_{\bp\in\bZ_{\ge 0}^n}
\sum_{q=0} ^{p_1+p_j} p_j\;\frac{p_1+p_j-q}{2}
\big[qp_2\cdots\widehat{p_j}\cdots p_n 
N_{g,n-1}(q,p_{N\setminus\{1,j\}})\big]
e^{-\la p,w\ra}
\\
=
\sum_{j=2} ^n \sum_{q=0} ^\infty 
\sum_{p_{N\setminus\{1,j\}}\in\bZ_{\ge 0} ^{n-2}}
\big[qp_2\cdots\widehat{p_j}\cdots p_n 
N_{g,n-1}(q,p_{N\setminus\{1,j\}})\big]
e^{-\la p_{N\setminus\{1,j\}}, w_{N\setminus\{1,j\}}\ra}
e^{-qw_1}
\\
\times
\sum_{\ell = 0} ^\infty
\ell e^{-2\ell w_1} \sum_{p_j=0} ^{q+2\ell} p_j e^{p_j(w_1-w_j)},
\end{multline*}
where the symbol $\widehat{\;\;}$
indicates omission of the variable, and
we set $p_1+p_j-q=2\ell$. Note that $N_{g,n}(p_N)=0$
unless $p_1+\cdots+p_n$ is even. Therefore, in the Laplace
transform we are summing
over all $p_N\in\bZ_{\ge0}^n$ such that $p_1+\cdots+p_n\equiv 0
\mod 2$.
Since 
$N_{g,n-1}(q,p_{N\setminus\{1,j\}})=0$ unless
$q+p_2+\cdots +\widehat{p_j}+\cdots +p_n\equiv 0 \mod 2$, 
only those $p_1, p_j$ and $q$ satisfying
$p_1+p_j-q\equiv 0\mod 2$ contribute in the summation. 
We now calculate from the last factor (the $p_j$-summation)
\begin{multline*}
\sum_{p_j=0} ^{q+2\ell}e^{-qw_1}\ell e^{-2\ell w_1}
 p_j e^{p_j(w_1-w_j)}
=
e^{-qw_1}\ell e^{-2\ell w_1}\frac{\partial}{\partial w_1}\;
\frac{e^{w_j}-e^{w_j} e^{(1+q+2\ell)(w_1-w_j)}}
{e^{w_j}-e^{w_1}}
\\
=
\frac{e^{-qw_1}\ell e^{-2\ell w_1}}{(e^{w_1}-e^{w_j})^2}
\Bigg[
e^{w_1+w_j}+(q+2\ell) e^{2w_1}e^{(q+2\ell)(w_1-w_j)}
-(1+q+2\ell)e^{w_1+w_j}e^{(q+2\ell)(w_1-w_j)}
\Bigg]
\\
=
\frac{1}{(e^{w_1}-e^{w_j})^2}
\Bigg[
e^{w_1+w_j}e^{-qw_1}\ell e^{-2\ell w_1}
 +\ell 
 (q+2\ell) e^{2w_1}e^{-qw_j}e^{-2\ell w_j}
 \\
-\ell(1+q+2\ell)e^{w_1+w_j}e^{-qw_j}e^{-2\ell w_j}
\Bigg]
\end{multline*} 
followed by the $\ell$-summation and then the $q$-summation.
We obtain
\begin{multline*}
=\sum_{j=2}^n 
\Bigg[
\frac{e^{w_1+w_j}}{(e^{w_1}-e^{w_j})^2}
\left(
\frac{\wL_{g,n-1}(w_{N\setminus\{j\}})}{(e^{w_1}-e^{-w_1})^2}
-
\frac{\wL_{g,n-1}(w_{N\setminus\{1\}})}{(e^{w_j}-e^{-w_j})^2}
\right)
\\
-\frac{e^{w_1}}{e^{w_1}-e^{w_j}}\;\frac{\partial}{\partial w_j}
\frac{\wL_{g,n-1}(w_{N\setminus\{1\}})}{(e^{w_j}-e^{-w_j})^2}
\Bigg].
\end{multline*}

The second line of (\ref{eq:integralrecursion}) contributes
\begin{multline*}
\sum_{j=2} ^n 
\sum_{\bp\in\bZ_{\ge 0}^n}
H(p_1-p_j)
\sum_{q=0} ^{p_1-p_j} p_j\;\frac{p_1-p_j-q}{2}
\big[qp_2\cdots\widehat{p_j}\cdots p_n 
N_{g,n-1}(q,p_{N\setminus\{1,j\}})\big]
e^{-\la p,w\ra}
\\
=
\sum_{j=2} ^n 
\sum_{\ell=0} ^\infty \ell e^{-2\ell w_1}
\sum_{p_j=0} ^\infty
p_j  e^{-p_j(w_1+w_j)}
\\
\times
\sum_{q=0} ^\infty e^{-qw_1}
\sum_{p_{N\setminus\{1,j\}}\in\bZ_{\ge 0} ^{n-2}}
\big[qp_2\cdots\widehat{p_j}\cdots p_n 
N_{g,n-1}(q,p_{N\setminus\{1,j\}})\big]
e^{-\la p_{N\setminus\{1,j\}},w_{N\setminus\{1,j\}}\ra}
\\
=
\sum_{j=2} ^n
\frac{e^{w_1+w_j}}{(1-e^{w_1+w_j})^2}\;
\frac{\wL_{g,n-1}(w_{N\setminus\{j\}})}
{(e^{w_1}-e^{-w_1})^2}.
\end{multline*}
In this calculation we set $p_1-p_j-q = 2\ell$.
Similarly, after putting 
$p_j-p_1-q=2\ell$, the third line of 
(\ref{eq:integralrecursion}) yields
\begin{multline*}
-\sum_{j=2} ^n 
\sum_{\bp\in\bZ_{\ge 0}^n}
 H(p_j-p_1)
\sum_{q=0} ^{p_j-p_1} p_j\;\frac{p_j-p_1-q}{2}
\big[qp_2\cdots\widehat{p_j}\cdots p_n 
N_{g,n-1}(q,p_{N\setminus\{1,j\}})\big]
e^{-\la p,w\ra}
\\
=
-\sum_{j=2} ^n 
\sum_{q=0} ^\infty 
\sum_{\ell=0} ^\infty 
\sum_{p_1=0} ^\infty
(p_1+q+2\ell) \ell  e^{-p_1(w_1+w_j)}
 e^{-2\ell w_j} e^{-qw_j}
\\
\times
\sum_{p_{N\setminus\{1,j\}}\in\bZ_{\ge 0} ^{n-2}}
\big[qp_2\cdots\widehat{p_j}\cdots p_n 
N_{g,n-1}(q,p_{N\setminus\{1,j\}})\big]
e^{-\la p_{N\setminus\{1,j\}},w_{N\setminus\{1,j\}}\ra}
\\
=
-\sum_{j=2} ^n
\frac{e^{w_1+w_j}}{(1-e^{w_1+w_j})^2}\;
\frac{\wL_{g,n-1}(w_{N\setminus\{1\}})}
{(e^{w_j}-e^{-w_j})^2}
+
\sum_{j=2} ^n
\frac{e^{w_1}}{e^{w_1}-e^{-w_j}}\;
\frac{\partial}{\partial w_j}\;
\frac{\wL_{g,n-1}(w_{N\setminus\{1\}})}
{(e^{w_j}-e^{-w_j})^2}.
\end{multline*}

Summing all contributions, we obtain
\begin{multline*}
\sum_{j=2}^n 
\Bigg[
\frac{e^{w_1+w_j}}{(e^{w_1}-e^{w_j})^2}
\left(
\frac{\wL_{g,n-1}(w_{N\setminus\{j\}})}{(e^{w_1}-e^{-w_1})^2}
-
\frac{\wL_{g,n-1}(w_{N\setminus\{1\}})}{(e^{w_j}-e^{-w_j})^2}
\right)
\\
-\frac{e^{w_1}}{e^{w_1}-e^{w_j}}\;\frac{\partial}{\partial w_j}
\frac{\wL_{g,n-1}(w_{N\setminus\{1\}})}{(e^{w_j}-e^{-w_j})^2}
\Bigg]
\\
+
\sum_{j=2} ^n
\frac{e^{w_1+w_j}}{(1-e^{w_1+w_j})^2}\;
\frac{\wL_{g,n-1}(w_{N\setminus\{j\}})}
{(e^{w_1}-e^{-w_1})^2}
\\
-\sum_{j=2} ^n
\frac{e^{w_1+w_j}}{(1-e^{w_1+w_j})^2}\;
\frac{\wL_{g,n-1}(w_{N\setminus\{1\}})}
{(e^{w_j}-e^{-w_j})^2}
+
\sum_{j=2} ^n
\frac{e^{w_1}}{e^{w_1}-e^{-w_j}}\;
\frac{\partial}{\partial w_j}\;
\frac{\wL_{g,n-1}(w_{N\setminus\{1\}})}
{(e^{w_j}-e^{-w_j})^2}
\\
=
\sum_{j=2} ^n \frac{\partial}{\partial w_j}
\left[
\left(
\frac{e^{w_1}}{e^{w_1}-e^{w_j}}-
\frac{e^{w_1+w_j}}{e^{w_1+w_j}-1}
\right)
\left(
\frac{
\wL_{g,n-1}(w_{N\setminus\{j\}})}
{(e^{w_1}-e^{-w_1})^2}
-
\frac{
\wL_{g,n-1}(w_{N\setminus\{1\}})}
{(e^{w_j}-e^{-w_j})^2}
\right)
\right].
\end{multline*}

To compute the Laplace transform of the fourth line of 
(\ref{eq:integralrecursion}), we note that
\begin{multline*}
\half \sum_{p_1=0} ^\infty \sum_{0\le q_1+q_2\le p_1}
q_1q_2(p_1-q_1-q_2)e^{-p_1w_1}f(q_1,q_2)
\\
=
\half \sum_{q_1=0}^\infty
\sum_{q_2=0}^\infty
\sum_{\ell=0} ^\infty
2\ell e^{-2\ell w_1}e^{-(q_1+q_2)w_1}q_1q_2f(q_1,q_2)
\\
=
\frac{1}{(e^{w_1}-e^{-w_1})^2}
\widehat{f}(w_1,w_1)
,
\end{multline*}
where we set $p_1-q_1-q_2=2\ell$, and
$$
\widehat{f}(w_1,w_2)=\sum_{q_1=0} ^\infty \sum_{q_2=0}
^\infty q_1q_2f(q_1,q_2) e^{-(q_1 w_1+q_2 w_2)}.
$$ 
The reason that  $p_1-q_1-q_2$ is even comes from
the fact that we are summing over $p_N\in\bZ_{\ge0}^n$
subject to $p_1+\cdots+p_n\equiv 0\mod 2$,
while on the fourth line of (\ref{eq:integralrecursion})
contributions vanish unless $q_1+q_2+p_2+\cdots+p_n
\equiv 0\mod 2$. Therefore, we can 
restrict the summation over those $p_1,q_1$ and
$q_2$ subject to $p_1\equiv q_1+q_2\mod 2$.
The calculation of the Laplace transform then becomes
straightforward, and the contribution is as in (\ref{eq:LTinw}).

To change from the $w$-coordinates to the $t$-coordinates,
we use (\ref{eq:t}) to find
$$
dw_j = \frac{2}{t_j ^2-1}\;dt_j,
\qquad
\frac{\partial}{\partial w_j} 
=\frac{t_j^2-1}{2}\;\frac{\partial}{\partial t_j}.
$$
Each factor changes as follows:
\begin{align*}
\frac{1}{(e^{w_j}-e^{-w_j})^2}
&=
\frac{1}{16}\;\frac{(t_j^2-1)^2}{t_j^2}
\\
\frac{e^{w_1}}{e^{w_1}-e^{w_j}}-
\frac{e^{w_1+w_j}}{e^{w_1+w_j}-1}
&=
\frac{t_j(t_1^2-1)}{t_1^2-t_j^2}
\\
\wL_{g,n}\big( w_1(t),\dots,w_n(t)\big)
&= 
(-1)^n2^{-n}\cL_{g,n}(t_1,\dots,t_n)(t_1^2-1)\cdots (t_n^2-1).
\end{align*}
We can now convert 
(\ref{eq:LTinw}) to (\ref{eq:LTrecursion})
by a straightforward calculation.
\end{proof}

We now prove Theorem~\ref{thm:LTK}.

\begin{thm}
The symmetric function $\wV_{g,n}^S(w_N)$
defined by the Laplace transform
  \begin{equation*}
 	\wV_{g,n}^S(w_1, \dots, w_n)dw_1\tensor \cdots
	\tensor dw_n 
	=
	d_1\tensor\cdots\tensor d_n 
	\int_{\bR_+ ^n}
	  v_{g,n}^S(\bp)
	  e^{-\la w,\bp \ra}dp_1\cdots dp_n
	 \end{equation*}
 satisfies the topological recursion 
 \begin{multline}
 \label{eq:wVrecursion}
\wV_{g,n}^S(w_N) 
= 
-2\sum_{j=2} ^\infty
\frac{\partial}{\partial w_j}
\left[
\frac{w_j}{w_1^2-w_j^2}
\left(
\frac{\wV_{g,n-1}^S(w_{N\setminus\{j\}})}{w_1 ^2}
-
\frac{\wV_{g,n-1}^S(w_{N\setminus\{1\}})}{w_j ^2}
\right)
\right]
\\
-
\frac{2}{w_1^2}
 \left(
  \wV_{g-1,n+1}^S(w_1,w_1, w_{N\setminus\{1\}}) 
  + \sum_{\substack{g_1 + g_2 = g, \\
  {I} \sqcup {J} = N\setminus\{1\}}}  
  \wV_{g_1,|I|+1}^S(w_1, w_{{I}}) 
  \wV_{g_2,|J|+1}^S(w_1, w_{{J}}) 
  \right).
\end{multline}
  \end{thm}

\begin{proof}
Since 
$$
\wV_{g,n} ^S(w_N) = (-1)^n 
\int_{\bR_+ ^n}p_1\cdots p_n
	  v_{g,n}^S(\bp)
	  e^{-\la w,\bp \ra}dp_1\cdots dp_n,
$$
we multiply both sides of  (\ref{eq:Srecursion}) by
$(-1)^n p_2\cdots p_n$ 
 and take
the Laplace transform. The left-hand side gives
$\wV_{g,n} ^S(w_N)$.

For a continuous function $f(q)$, by putting 
$p_1+p_j-q=\ell$,  we have
\begin{multline}
\label{eq:LT1}
\int_0 ^\infty \!\!dp_1\int_0 ^\infty \!\!dp_j
\int_0 ^{p_1+p_j}\!\!dq \; p_jq(p_1+p_j-q)f(q)
e^{-(p_1w_1+p_jw_j)}
\\
=
\int_0 ^\infty \!\!dq\int_0 ^\infty \!\!d\ell
\int_0 ^{q+\ell} \!\! dp_j\;q\ell f(q) e^{-qw_1}e^{-\ell w_1}
p_je^{p_j(w_1-w_j)}
\\
=
\frac{1}{(w_1-w_j)^2}
\int_0 ^\infty \!\!dq\int_0 ^\infty \!\!d\ell\;q\ell f(q)
\bigg[e^{-(q+\ell)w_1}-e^{-(q+\ell)w_j}
+
(q+\ell)(w_1-w_j)e^{-(q+\ell)w_j}\bigg]
\\
=
\frac{1}{(w_1-w_j)^2}
\left(
\frac{\widehat{f}(w_1)}{w_1 ^2}
-
\frac{\widehat{f}(w_j)}{w_j ^2}
\right)
-\frac{1}{w_1-w_j}\;\frac{\partial}{\partial w_j}\left(
\frac{\widehat{f}(w_j)}{w_j^2}\right)
\\
=
\frac{\partial}{\partial w_j}
\left[
\frac{1}{w_1-w_j}
\left(
\frac{\widehat{f}(w_1)}{w_1 ^2}
-
\frac{\widehat{f}(w_j)}{w_j ^2}
\right)
\right]
,
\end{multline}
where $\widehat{f}(w)=\int_0 ^\infty qf(q)e^{-qw}dq$.
By setting $p_1-p_j-q=\ell$ we calculate
\begin{multline}
\label{eq:LT2}
\int_0 ^\infty \!\!dp_1\int_0 ^\infty \!\!dp_j H(p_1-p_j)
\int_0 ^{p_1-p_j}\!\!dq \; p_jq(p_1-p_j-q)f(q)
e^{-(p_1w_1+p_jw_j)}
\\
=
\int_0 ^\infty \!\!dq\int_0 ^\infty \!\!d\ell
\int_0 ^{\infty} \!\! dp_j\;q\ell f(q) e^{-qw_1}e^{-\ell w_1}
p_je^{-p_j(w_1+w_j)}
=
\frac{1}{(w_1+w_j)^2}\;
\frac{\widehat{f}(w_1)}{w_1 ^2}, 
\end{multline}
and similarly, 
\begin{multline}
\label{eq:LT3}
-\int_0 ^\infty \!\!dp_1\int_0 ^\infty \!\!dp_j H(p_j-p_1)
\int_0 ^{p_j-p_1}\!\!dq \; p_jq(p_j-p_1-q)f(q)
e^{-(p_1w_1+p_jw_j)}
\\
=
-\int_0 ^\infty \!\!dq\int_0 ^\infty \!\!d\ell\;
\int_0 ^{\infty} \!\! dp_1\;q\ell f(q) e^{-qw_j}e^{-\ell w_j}
(p_1+q+\ell)e^{-p_1(w_1+w_j)}
\\
=
-\int_0 ^\infty \!\!dq\int_0 ^\infty \!\!d\ell\;
q\ell f(q) e^{-qw_j}e^{-\ell w_j}
\left[
\frac{1}{(w_1+w_j)^2}+\frac{q+\ell}{w_1+w_j}
\right]
\\
=
-\frac{1}{(w_1+w_j)^2}\;
\frac{\widehat{f}(w_j)}{w_j ^2}
+
\frac{1}{w_1+w_j}\;\frac{\partial}{\partial w_j}\left(
\frac{\widehat{f}(w_j)}{w_j^2}\right).
\end{multline}
Adding (\ref{eq:LT2}) and (\ref{eq:LT3}) we obtain
\begin{multline*}
\int_0 ^\infty \!\!dp_1\int_0 ^\infty \!\!dp_j H(p_1-p_j)
\int_0 ^{p_1-p_j}\!\!dq \; p_jq(p_1-p_j-q)f(q)
e^{-(p_1w_1+p_jw_j)}
\\
-\int_0 ^\infty \!\!dp_1\int_0 ^\infty \!\!dp_j H(p_j-p_1)
\int_0 ^{p_j-p_1}\!\!dq \; p_jq(p_j-p_1-q)f(q)
e^{-(p_1w_1+p_jw_j)}
\\
=
\frac{1}{(w_1+w_j)^2}
\left(
\frac{\widehat{f}(w_1)}{w_1 ^2}
-
\frac{\widehat{f}(w_j)}{w_j ^2}
\right)
+
\frac{1}{w_1+w_j}\;\frac{\partial}{\partial w_j}\left(
\frac{\widehat{f}(w_j)}{w_j^2}\right)
\\
=
-\frac{\partial}{\partial w_j}
\left[
\frac{1}{w_1+w_j}
\left(
\frac{\widehat{f}(w_1)}{w_1 ^2}
-
\frac{\widehat{f}(w_j)}{w_j ^2}
\right)
\right].
\end{multline*}
The sum of the right-hand sides
of  (\ref{eq:LT1})-(\ref{eq:LT3}) thus becomes
$$
\frac{\partial}{\partial w_j}
\left[
\left(
\frac{1}{w_1-w_j}-\frac{1}{w_1+w_j}
\right)
\left(
\frac{\widehat{f}(w_1)}{w_1 ^2}
-
\frac{\widehat{f}(w_j)}{w_j ^2}
\right)
\right].
$$
Therefore, the first three lines of (\ref{eq:Srecursion})
yield
\begin{equation*}
-2\sum_{j=2} ^\infty
\frac{\partial}{\partial w_j}
\left[
\frac{w_j}{w_1^2-w_j^2}
\left(
\frac{\wV_{g,n-1}^S(w_{N\setminus\{j\}})}{w_1 ^2}
-
\frac{\wV_{g,n-1}^S(w_{N\setminus\{1\}})}{w_j ^2}
\right)
\right].
\end{equation*}
For a continuous function $f(q_1,q_2)$, we have
\begin{multline*}
\int_0^\infty\!\!dp_1
\int\!\!\int_{0\le q_1+q_2\le p_1}
\;
q_1q_2(p_1-q_1-q_2)f(q_1,q_2)e^{-p_1w_1}dq_1dq_2
\\
=
\int_0^\infty \!\!dq_1\int_0^\infty \!\!dq_2\int_0^\infty \!\!d\ell
\; q_1q_2\ell f(q_1,q_2)e^{-\ell w_1}e^{-(q_1+q_2)w_1}
=
\frac{\widehat{f}(w_1,w_1)}{w_1 ^2},
\end{multline*}
where $\widehat{f}(w_1,w_2)=\int_{\bR_+^2}p_1p_2f(p_1,p_2)
e^{-(p_1w_1+p_2w_2)}dp_1dp_2$.
Thus the last two lines of (\ref{eq:Srecursion})
give
$$
-\frac{2}{w_1^2}
 \left(
  \wV_{g-1,n+1}^S(w_1,w_1, w_{N\setminus\{1\}}) 
  + \sum_{\substack{g_1 + g_2 = g, \\
  {I} \sqcup {J} = N\setminus\{1\}}}  
  \wV_{g_1,|I|+1}^S(w_1, w_{{I}}) 
  \wV_{g_2,|J|+1}^S(w_1, w_{{J}}) 
  \right).
$$
This completes the proof of (\ref{eq:wVrecursion}).
\end{proof}

Let us now change the coordinates from $w_j$'s to 
$t_j$'s that are given by
$$
w_j=-\frac{2}{t_j}
$$
this time. 
This change of coordinate gives
$$
dw_j=\frac{2}{t_j^2}dt_j, \qquad \frac{\partial}{\partial w_j}
=\frac{t_j^2}{2}\;\frac{\partial}{\partial t_j}.
$$
 Thus the relation in terms of symmetric differential form is
$$
V_{g,n}^S(t_N) dt_N = \wV_{g,n}^S(w_N) dw_N,
$$
or 
$$
V_{g,n}^S(t_N) =2^n \frac{\wV_{g,n}^S(w_N)}
{t_1^2\cdots t_n^2}.
$$
So we multiply
both sides of (\ref{eq:wVrecursion}) by
 $\frac{2^n}
{t_1^2\cdots t_n^2}$.
From the first term of the first line we obtain
\begin{multline*}
-2\frac{2^n}
{t_1^2\cdots t_n^2}
\sum_{j=2} ^\infty
\frac{\partial}{\partial w_j}
\left[
\frac{w_j}{w_1^2-w_j^2}
\;
\frac{\wV_{g,n-1}^S(w_{N\setminus\{j\}})}{w_1 ^2}
\right]
\\
=
-2\sum_{j=2} ^\infty\frac{\partial}{\partial t_j}
\left[
\half\;\frac{t_1^2\;t_j}{t_1^2-t_j^2}
\;\frac{t_1^2}{4}\;
V_{g,n-1}^S(t_{N\setminus\{j\}})
\right]
=
-\frac{1}{4}\sum_{j=2} ^\infty\frac{\partial}{\partial t_j}
\left[
\frac{t_j}{t_1^2-t_j^2}
\;t_1^4\;
V_{g,n-1}^S(t_{N\setminus\{j\}})
\right].
\end{multline*}
Similarly, the second term of the first line becomes 
\begin{multline*}
2\frac{2^n}
{t_1^2\cdots t_n^2}
\sum_{j=2} ^\infty
\frac{\partial}{\partial w_j}
\left[
\frac{w_j}{w_1^2-w_j^2}
\;
\frac{\wV_{g,n-1}^S(w_{N\setminus\{1\}})}{w_j ^2}
\right]
\\
=
2\sum_{j=2} ^\infty\frac{\partial}{\partial t_j}
\left[
\half\;\frac{t_1^2\;t_j}{t_1^2-t_j^2}
\;\frac{t_j^2}{4}\;
\frac{t_j^2 
V_{g,n-1}^S(t_{N\setminus\{1\}})
}{t_1^2}
\right]
=
\frac{1}{4}\sum_{j=2} ^\infty\frac{\partial}{\partial t_j}
\left[
\frac{t_j}{t_1^2-t_j^2}
\;t_j^4\;
V_{g,n-1}^S(t_{N\setminus\{1\}})
\right].
\end{multline*}
The second line of (\ref{eq:wVrecursion}) is easy to convert.
This completes the proof of Theorem~\ref{thm:LTK}.

\section{Examples}
\label{app:examples}

For $(g,n)=(0,3)$, there are three topological 
shapes of ribbon graphs
listed in Figure~\ref{fig:03}. Cyclic permutations of
$(p_1,p_2,p_3)$ produce different graphs.

\begin{figure}[htb]
\centerline{\epsfig{file=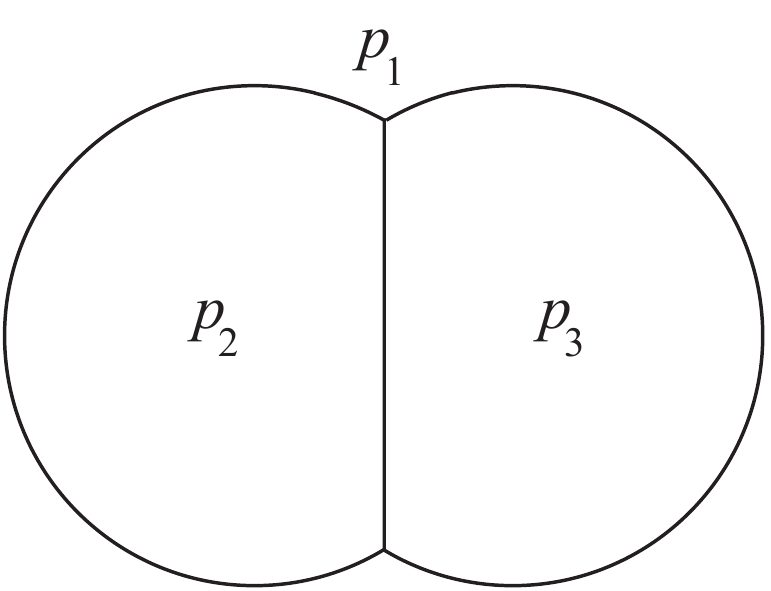, width=1.5in}}
\vskip0.2in
\centerline{
\epsfig{file=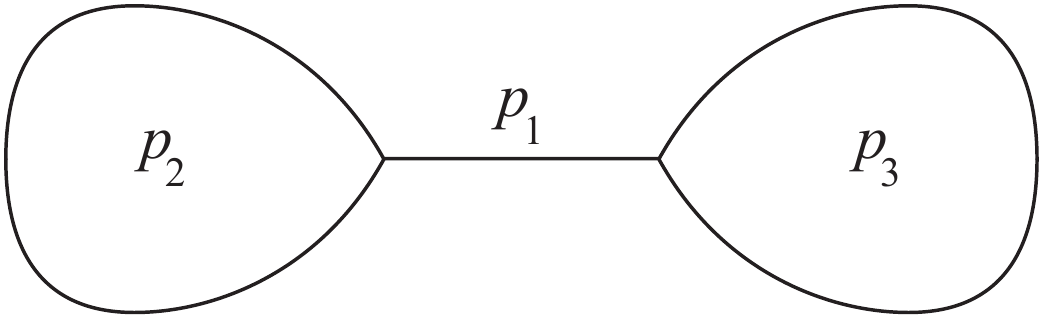, height=0.8in}
\hskip0.2in
\epsfig{file=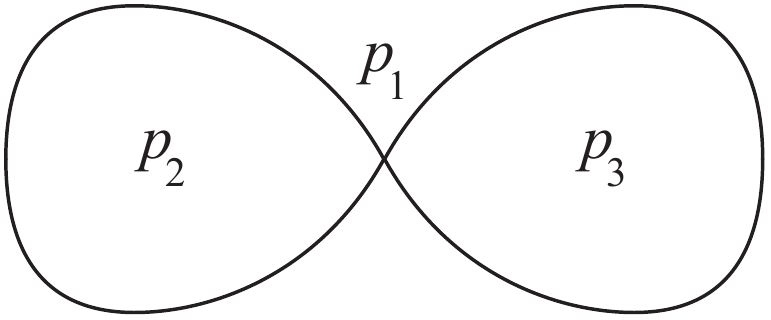, height=0.8in}}
\caption{Three ribbon graphs for $(g,n)=(0,3)$.}
\label{fig:03}
\end{figure}

Which ribbon graph corresponds to a point
$(p_1,p_2,p_3)\in\bZ_+ ^3$ depends on which inequality
these three numbers satisfy. If $p_1>p_2+p_3$, then the 
dumbbell shape (Figure~\ref{fig:03},
bottom left) corresponds to this point. If 
$p_1=p_2+p_3$, then the shape of $\infty$ 
(Figure~\ref{fig:03},
bottom right)
corresponds, 
and if no coordinate is greater than the sum of the other two, 
then the double circle graph (Figure~\ref{fig:03},
top)
corresponds. These inequalities 
divide $\bZ_+ ^3$ into four regions as in Figure~\ref{fig:p03}.

\begin{figure}[htb]
\centerline{\epsfig{file=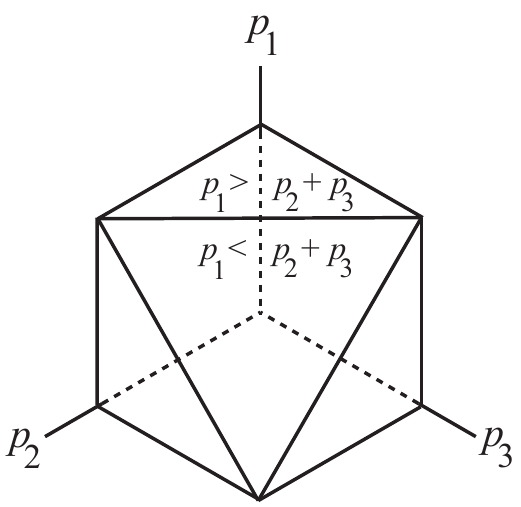, width=1.8in}}
\caption{The partition of $\bZ_+ ^3$.}
\label{fig:p03}
\end{figure}

Thus we conclude 
\begin{equation}
\label{eq:N03}
N_{0,3}(p_1,p_2,p_3)=
\begin{cases}
1\qquad p_1+p_2+p_3\equiv 0\mod 2,\\
0\qquad \text{otherwise.}
\end{cases}
\end{equation}
The even parity condition can be met if all three are even
or only one of them is even. Let us substitute 
$p_j=2q_j$ when it is even and $p_j=2q_j-1$ if it is odd.
Thus the Laplace transform can be calculated by
\begin{multline*}
L_{0,3}(w_1,w_2,w_3) = \sum_{(p_1,p_2,p_3)
\in \bZ_{+}^3} N_{0,3}(p_1,p_2,p_3)
e^{-(p_1w_1+p_2w_2+p_3w_3)}
\\
=
\sum_{(q_1,q_2,q_3)
\in \bZ_{+}^3} 
\bigg(
1
+e^{w_1+w_2}+e^{w_2+w_3}+e^{w_3+w_1}
\bigg)
e^{-2(q_1w_1+q_2w_2+q_3w_3)}
\\
=
\bigg(
1
+e^{w_1+w_2}+e^{w_2+w_3}+e^{w_3+w_1}
\bigg)
\frac{e^{-(w_1+w_2+w_3)}}
{(e^{w_1}-e^{-w_1})(e^{w_2}-e^{-w_2})(e^{w_3}-e^{-w_3})}.
\end{multline*}
Using $e^{-w_j} = \frac{t_j+1}{t_j-1}$, we obtain
\begin{equation}
\label{eq:F03}
L_{0,3}\big(w(t_1),w(t_2),w(t_3)\big)
=
-\frac{1}{16}\;(t_1+1)(t_2+1)(t_3+1)\left(
1+\frac{1}{t_1t_2t_3}
\right)
\end{equation}
and 
\begin{equation}
\label{eq:L03}
\cL_{0,3}(t_1,t_2,t_3)=
\frac{\partial^3}{\partial t_1\partial t_2\partial t_3}
\;L_{0,3}\big(w(t_1),w(t_2),w(t_3)\big)
=
-\frac{1}{16}
\left(
1-\frac{1}{t_1 ^2\;t_2^ 2\;t_3^2}
\right).
\end{equation}

For $(g,n)=(1,1)$, there are two ribbon graphs 
(see Figure~\ref{fig:11}) corresponding to a
hexagonal and a square tiling of the plane. The hexagonal 
tiling gives a ribbon graph on the left, and the
square one on the right 
is a degeneration  obtained by shrinking the
horizontal edge to $0$. The automorphism group
is $\bZ/6\bZ$ for the degree $3$ graph, and $\bZ/4\bZ$
for the degree $4$ graph.

\begin{figure}[htb]
\centerline{\epsfig{file=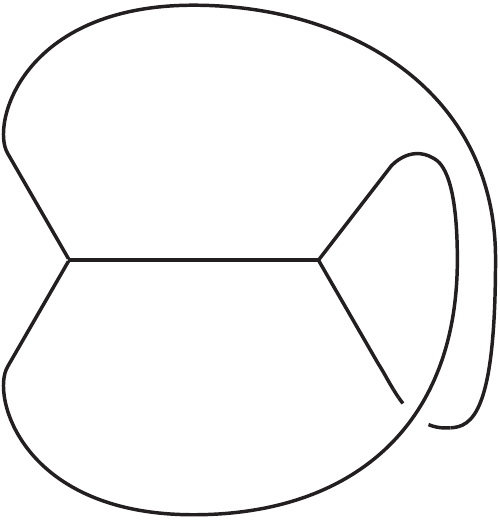, height=1in}
\hskip1in
\epsfig{file=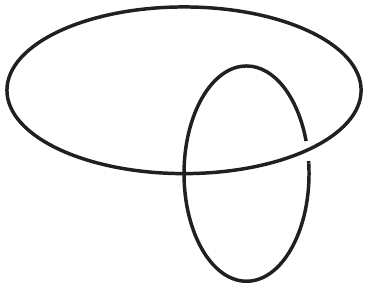, height=0.8in}
}
\caption{Two ribbon graphs of type $(1,1)$.}
\label{fig:11}
\end{figure}

The number of integral ribbon graphs in this case is
the number of partitions of the half of the
 given perimeter length $p=2q\in 2\bZ_+$ into two or three
 positive integers corresponding to edge lengths. 
 Taking the automorphism factors into account, we 
 calculate 
 $$
 N_{1,1}(2q)
 = \frac{1}{4} (q-1)+\frac{1}{6} \sum_{r=1} ^{q-1}(r-1)
 =\frac{1}{12}\;(q^2-1).
 $$
 Therefore,
\begin{equation}
\label{eq:N11}
N_{1,1}(p)=
\begin{cases}
\frac{1}{48}\;(p^2-4)\qquad p\equiv 0\mod 2,\\
0\qquad \text{otherwise.}
\end{cases}
\end{equation}
The Laplace transform can be calculated immediately:
$$
L_{1,1}(w) = \sum_{p=2} ^\infty N_{1,1}(p)e^{-pw}
=
\frac{1}{12}\sum_{q=1} ^\infty (q^2-1)e^{-2qw}
=\frac{3e^{2w}-1}{12(e^{2w}-1)^3}.
$$
We thus obtain
\begin{equation}
\begin{aligned}
\label{eq:F11}
L_{1,1}\big(w(t)\big)&=-\frac{1}{384}\;
\frac{(t+1)^4}{t^2}\;\left(t-4+\frac{1}{t}\right)
\end{aligned}
\end{equation}
and
\begin{equation}
\label{eq:cL11}
\cL_{1,1}(t)=-\frac{1}{2^7}\;\frac{(t^2-1)^3}{t^4}.
\end{equation}

The values of $(g,n)$ corresponding to genus $0$
unstable geometries
$(0,1), (0,2)$ play an
important role in topological recursion. We have seen this
phenomena 
in Hurwitz theory \cite{EMS, MZ}. Let us
 consider the unstable $(0,2)$ case for the integral
ribbon graph counting. Although we have restricted our
ribbon graphs to have vertices of degree $3$ or more, 
it is indeed more consistent to allow vertices of degree $2$.
Actually, a \emph{metric} ribbon graph
with integer edge lengths is a plain ribbon graph whose 
vertices have  degree $2$ or more. For such a ribbon graph,
we assign length $1$ to every edge. We recover a
metric ribbon graph with integer edge length by
disregarding all
vertices of degree $2$. This point of view is natural 
when we assign a Belyi morphism \cite{Belyi, MP1998, SL}
to a ribbon graph.

\begin{figure}[htb]
\centerline{\epsfig{file=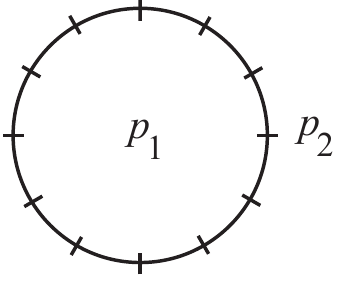, width=1.2in}}
\caption{A ribbon graph of type $(0,2)$.}
\label{fig:02}
\end{figure}

Once we allow degree $2$ vertices, there is only one kind of
ribbon graph of type $(0,2)$, which is a circle (Figure~\ref{fig:02}).
If the graph has $p$ edges, then the automorphism group of
this graph is $\bZ/p\bZ$. Therefore, we have
\begin{equation}
\label{eq:N02}
N_{0,2}(p_1,p_2)=\frac{1}{p_1}\delta_{p_1p_2},
\end{equation}
and its Laplace transform becomes
$$
L_{0,2}(w_1,w_2) = \sum_{p_1=1} ^\infty \sum_{p_2=1} ^\infty
N_{0,2}(p_1,p_2) e^{-(p_1w_1+p_2w_2)}
=-\log(1-e^{-w_1-w_2}).
$$
In terms of the $t$-coordinates we have
\begin{equation}
\label{eq:F02}
L_{0,2}\big(w(t_1),w(t_2)\big) = 
-\log\left(1-\frac{t_1+1}{t_1-1}\cdot 
\frac{t_2+1}{t_2-1}\right),
\end{equation}
which gives 
\begin{equation}
\label{eq:L02}
\cL_{0,2}(t_1,t_2)=\frac{1}{(t_1+t_2)^2}.
\end{equation}
Note that (\ref{eq:L02}) is not a Laurent polynomial and
exhibits
  an exception to the general
statement of Theorem~\ref{thm:LTrecursion}.

The parametrization of the spectral curve (\ref{eq:spectral})
defines the $x$-projection map
\begin{equation*}
\pi:\bP^1 \owns t \longmapsto x=\frac{t+1}{t-1}+\frac{t-1}{t+1}
\in\bP^1.
\end{equation*} 
We find that  the difference
of the Cauchy differentiation kernels of the curve upstairs and
downstairs is 
$\cL_{0,2}(t_1,t_2)$:
\begin{equation}
\label{eq:02=B}
\cL_{0,2}(t_1,t_2)dt_1 \tensor dt_2=
\frac{dt_1 \tensor dt_2}{(t_1+t_2)^2}
=
\frac{dt_1\tensor dt_2}{(t_1-t_2)^2}-
\pi^*\frac{dx_1\tensor dx_2}{(x_1-x_2)^2},
\end{equation}
where $\pi$ is the $x$-projection map  (\ref{eq:xprojection}).
We note that this situation is exactly the same as the
Hurwitz theory \cite[Remark 4.5]{EMS}.

The other genus $0$ unstable case $(0,1)$ is 
important because it identifies the embedding of 
the spectral curve (\ref{eq:spectral}) in $\bC^2$.
It is also subtle
 because we
need to allow degree $1$ vertices. 
Since all possible trees can
be included if we allow degree $1$ vertices, 
we have to make
a choice as to what kind of trees are allowed. 
For example,
 we could allow arbitrary trees as in the 
Hurwitz theory  \cite{OP1}. In the 
current 
integral ribbon graph case, we need to make a more restrictive choice.
Since this topic has no direct relevance to the 
main theorems of this paper, it will be 
treated elsewhere.

Using the recursion formula (\ref{eq:LTrecursion}), 
we can calculate $\cL_{g,n}(t_N)$ systematically. 
A few examples are listed below.
\begin{equation}
\label{eq:L04}
\cL_{0,4}(t_1,t_2,t_3,t_4)
=
\frac{1}{2^8} 
\Bigg[
3
\sum_{j=1} ^4 t_j^2
-9 
-\sum_{i<j}\frac{1}{t_i^2 \;t_j^2}
-\frac{9}{t_1^2\;t_2^2\;t_3^2\;t_4^2}
+\frac{3}{t_1^2\;t_2^2\;t_3^2\;t_4^2}
\sum_{j=1} ^4\frac{1}{t_j^2}
   \Bigg].
\end{equation}
\begin{multline}
\label{eq:L12}
\cL_{1,2}(t_1,t_2)
=
\frac{1}{2^{11}}
\Bigg[
5
\left(
t_1^4+t_2^4
\right)
+3t_1^2\;t_2^2
-18
\left(
t_1^2+t_2^2
\right)
+27 
-4
\left(
\frac{1}{t_1^2}+\frac{1}{t_2^2}
\right)
\\
+\frac{27}{t_1^2\;t_2^2}
-\frac{18}{t_1^2\;t_2^2}
\left(
\frac{1}{t_1^2}+\frac{1}{t_2^2}
\right)
+\frac{3}{t_1^4\;t_2^4}
+\frac{5}{t_1^2\;t_2^2}
\left(
\frac{1}{t_1^4}+\frac{1}{t_2^4}
\right)
\Bigg].
\end{multline}
\begin{equation}
\label{eq:L21}
\cL_{2,1}(t)
=
-\frac{21}{2^{19}}\;
\frac{(t^2-1)^7}{t^{8}}
\left(
5\; t^2+ 6 +\frac{5}{t^2}
\right).
\end{equation}
\begin{equation}
\label{eq:L31}
\cL_{3,1}(t)=-\frac{11}{2^{30}}\;\frac{(t^2-1)^{11}}{t^{12}}
\left(
2275\;t^4+4004\; t^2+4722 +\frac{4004}{t^2}+\frac{2275}{t^4}
\right).
\end{equation}

\end{appendix}

%Bibliography

\providecommand{\bysame}{\leavevmode\hbox to3em{\hrulefill}\thinspace}

\bibliographystyle{amsplain}

\end{document}